\documentclass[11pt,a4paper,reqno]{article}

\usepackage{template}
\usepackage{mathrsfs}
\usepackage{graphicx}
\pgfmathsetseed{123}

\newif\iffinal
\finaltrue

  \newcounter{kconstant}
  \newcommand{\nck}[1]{\refstepcounter{kconstant}\label{#1}}
  \newcommand{\uck}[1]{k_{\ref{#1}}}
\begin{document}

\title{Random Walks on Dynamical Random Environments with Non-Uniform Mixing}
\date{\today}
\author{Oriane Blondel\thanks{Univ Lyon, CNRS, Universit\'e Claude Bernard Lyon 1, UMR 5208, Institut Camille Jordan, F--69622, Villeurbanne, France.}
  \and
  Marcelo R. Hil\'ario \thanks{UFMG, Dep.\ de Matem\'atica, Av.\ Ant\^onio Carlos 6627, CP 702, 31270-901, Belo Horizonte, MG - Brazil.}
  \and
  Augusto Teixeira \thanks{IMPA, Estrada Dona Castorina 110, 22460-320 Rio de Janeiro, RJ - Brazil}}

\maketitle

\begin{abstract}
In this paper we study random walks on dynamical random environments in $1 + 1$ dimensions.
Assuming that the environment is invariant under space-time shifts and fulfills a mild mixing hypothesis, we establish a law of large numbers and a concentration inequality around the asymptotic speed.
The mixing hypothesis imposes a polynomial decay rate of covariances on the environment with sufficiently high exponent but does not impose uniform mixing. 
Examples of environments for which our methods apply include the contact process and Markovian environments with a positive spectral gap, such as the East model.
For the East model we also obtain that the distinguished zero satisfies a Law of Large Numbers with strictly positive speed.
\end{abstract}

\bigskip

Mathematics Subject Classification (2010): 60K35, 82B43, 60G55

\renewcommand\footnotemark{}
\renewcommand\footnoterule{}
\let\thefootnote\relax\footnotetext{{\bf Keywords:} Random walks, dynamic random environments, strong law of large numbers, renormalization}

\section{Introduction}

The research on random walks in random media finds its motivation in various questions, ranging from ecology, chemistry, particle physics and pure mathematics.

Within probability, the study of such processes started with the very interesting case of a random walk on a one-dimensional static media which has already been throughly studied, see for instance \cite{MR0380998, MR657919, MR0362503}.
Understanding the higher dimensional cases remains a great challenge despite of important progress obtained in that direction as well.
We refer to \cite{MR2198849} for a survey on the subject.

Besides the case of static media, substantial effort has been dedicated to the investigation of random walks on dynamical random environments.
The main results in this field depend a great deal on the specific dynamics under consideration as we discuss below.

There are several works concerned with quenched invariance principles for random walks on environments that change independently at each time step, see for instance \cite{MR1477651}.
Another important class of examples considered in the literature is the so-called strong-mixing environments, see \cite{comets2004, MR2191225, MR2786643}.
For random walks defined on this type of environment and under certain conditions such as, the cone-mixing property, it is possible to extract an approximate renewal structure yielding a Law of Large Numbers (LLN) and, in some cases, a Central Limit Theorem.
It is worth noticing that these conditions are usually quite restrictive, meaning that they require that the environment mixes either very fast or uniformly on the initial configuration.
For this reason, the techniques developed for the study of these processes do not seem to apply easily for most of the examples of random environments that we present in this paper.

Another technique in the field consists in analyzing the random environment as viewed from the walker, see \cite{redig2013, MR3580037}.
This technique provides proofs of Law of Large Numbers and Central Limit Theorems under somehow weaker mixing conditions like polynomial mixing rate.
However the mixing should be essentially uniform over the initial configuration.

\medskip
Cases in which the random environment is not assumed to mix uniformly have also been studied.
In \cite{MR2123642} and \cite{Avena2017} instances of such environments presenting a positive spectral gap  where considered.
However, these studies are perturbative in the sense that the environment seen from the random walker needs to be close to a process which has the same invariant measure as the environment itself. 
This includes cases of weak interaction between the walker and the environment which had been previously studied in various contexts.

Another example of environment that present non-uniform mixing is the contact process.
Laws of Large Numbers and Central Limit Theorems where proved for random walks on this environment in \cite{denhollander2014, Bethuelsen2018, mountford2015}.
Again, in these papers, the techniques used seem to be reasonably dependent on the specific environment under consideration.

A challenging type of environments is given by the conservative particle systems for conservation of particles implies in poor mixing rates which complicates the application of standard methods commonly used in the strong mixing case.
Examples of these environments include the exclusion process \cite{MR3108811, MR3407222} and Poissonian fields of random walks \cite{dHKS14, zbMATH06514478, BHSST1, BHSST2}.
For random walkers on these environments Laws of Large Numbers and  Central Limit Theorems can be achieved under the hypothesis of strong drift or for some perturbative regimes.
Using similar methods the evolution of the front of an infection process on a similar environment can be studied \cite{MR3531680}.
In each of the papers cited in the present paragraph, the proofs fit very specifically to the model in question, and do not seem to be easily adapted to other environments.

In the present paper we develop a robust framework that can be applied in a rather simple manner to prove a law of large numbers for random walks on several one-dimensional dynamical random environments.
Roughly speaking, for a broad class of such processes, one just has to check that the environment satisfies a simple space-time mixing inequality in order to be able to apply our results.
As an example of the applicability of our techniques, in Section \ref{s:applications} the validity of this condition is verified in a simple way for several important environments.

We consider random walks evolving  on dynamical random environments in dimension one sometimes called $1 + 1$ to account for the time dimension as well.
We will assume that the environment is invariant with respect to space-time shifts and moreover that it satisfies the following mixing property:
\begin{display}
  \label{e:decouple}
  for any pair of space-time boxes $B_1$ and $B_2$, with side length $5r$ and mutual distance at least $r$ and any pair of events $A_1$ and $A_2$ that only depend on the random environment inside $B_1$ and $B_2$ respectively, we have $\Cov(A_1, A_2) \leq c  r^{-\alpha}$.
\end{display}
See Assumption~\ref{hyp:invariantenv} and Definition~\ref{d:decouple} below for the precise statements of our assumptions.
Above, $\Cov$ stands for the covariance with respect to the law of the environment and $c$ is just a positive constant.
Note that this condition implies ergodicity of the process under time shifts, however it does not imply uniformity of the mixing.
As it will become clear below, our methods will work as soon as the exponent $\alpha$ appearing in \eqref{e:decouple} is sufficiently large ($\alpha > 8$ is enough for proving a LLN).

Suppose that on top of a translation invariant environment satisfying \eqref{e:decouple} we start a continuous-time, nearest-neighbor random walk whose jumps depend locally on the state of the environment immediately before their occurrence.
For now, let us assume that the jumping times of the random walk are given by a Poisson process with unitary rate, which is independent of the underlying environment, although this condition will be relaxed in Section~\ref{ss:randomwalker}.
Our main result, Theorem~\ref{t:main}, states that there exists $v \in [-1, 1]$ such that
\begin{equation}
  \lim_{t \to \infty} \frac{X_t}{t} = v, \qquad \text{almost surely}.
\end{equation}
Moreover we obtain some concentration bounds for $X_t/t$ around $v$, see \eqref{e:ldb}.

Another interesting result we present in this paper concerns random walks that can only move to one side on the set of the integers.
In this situation, we show that under the condition \eqref{e:decouple} with $\alpha > 8.5$,
\begin{display}
\label{e:speed_front}
  if the random walker can only jump to the right and has a positive probability of jumping within one time unit then $v$ is strictly positive.
\end{display}
See Theorem~\ref{t:positive} for a more detailed statement.

Besides being non-perturbative, our methods mainly require that the  environment satisfy a mixing hypothesis which is not as restrictive as the usual strong mixing or uniform mixing conditions previously considered, e.g.\ the cone-mixing condition \cite{MR2786643}.
To exemplify the generality of our methods, in Section~\ref{s:applications} we provide several examples of processes that fall into our hypotheses.
These include the contact process and particle systems with positive spectral gap, such as independent spin-flip dynamics, the East model \cite{eastphys} and the FA-1f model \cite{FA1}.
We also present an application for random walkers evolving on a set of renewal chains, as introduced in \cite{zbMATH06514478}.

For some of the models above, Laws of Large Numbers and sometimes Central Limit Theorems have been proved before in the literature by several different methods combining renormalisation, regeneration times and analysis of the random environment as seem from the random walk.
However, as far as we are aware, the results for the East model and the FA-1f model are new.
An interesting consequence of \eqref{e:speed_front} is that the distinguished zero of the East model satisfies a Law of Large Numbers with strictly positive speed, see Section~\ref{ss:east}.

\begin{remark}
  We believe that the techniques presented in this article should be naturally adapted to the discrete time framework.
  However, the hypothesis that two random walk trajectories starting at different points in space cannot cross each other is vital in our argumentation, see \eqref{e:monotone}.
\end{remark}

\begin{remark}
Our methods do not provide a Central Limit Theorem for the random walk.
It seems that some new ideas will be needed in order to achieve this goal.
We leave as an interesting future question to establish limiting results for the fluctuations of the walker around its expected position under general mixing hypotheses on the environment.
\end{remark}

\begin{remark}
One may be tempted to think that there might exist a simpler proof of the Law of Large Numbers stated in our main theorem using exclusively some type of ergodicity argument.
However, we would like to draw the reader's attention to an example presented in Section~\ref{s:example} of random walk naturally defined on an ergodic space-time environment that does not satisfy a LLN.
\end{remark}

\section{Mathematical setting and main results}
\label{s:notation}

Throughout the text, $c$ and $c'$ will denote positive constants whose values are allowed to change from line to line.
All constants may depend on the random environment (in particular on $\alpha$) and on the evolution rules of the random walk.
Further dependence will be made explicit, for example, we write $c = c(\varepsilon)$ to refer to a constant that depends on $\varepsilon$ and possibly on the law of the random environment and the evolution rules of the random walk.
Numbered constants such as $c_0, c_1, \ldots$ and $k_0, k_1, \ldots$ stand for positive numbers whose value is fixed at their first appearance in the text.

As mentioned before, we consider continuous-time random walks that evolve in discrete space $\mathbb{Z}$.
Its position at a given time is an element of the set 
\[
\mathbb{L} := \mathbb{Z} \times \mathbb{R}_+.
\]
The evolution of the random walk depends locally on the value taken by the dynamic environment around its current position.
That is, the distribution of each jump depends on the environment restricted to a bounded region of the environment around the position of the walker just before the jump.
The kind of environments that we consider are described in Section \ref{ss:environment} and the jumping rules will be given in Section \ref{ss:randomwalker}.

\subsection{Environment}
\label{ss:environment}

In this paper the environment will be given by random functions $(x,t)\mapsto\eta_t(x)$, for $x \in \mathbb{Z}$ and $t \in \mathbb{R}_+$, where $\eta_t(x)$ takes value in a countable state space $S$ and represents the state of site $x$ at time $t$.
Although our techniques apply in more general context, for most of the examples we consider, $S$ will be either $\{0, 1\}$ (such as in the case of the contact process and the East model -- see, Sections \ref{ss:contact} and \ref{ss:east} respectively), $\{-1, 1\}$ (for the Glauber dynamics of the Ising model -- see, Section \ref{ss:spectral}) or the set of natural numbers $\mathbb{N}$ (as in the example of the renewal environment -- see Section \ref{ss:renewal}).
We denote $\eta_t := (\eta_t(x))_{x\in\mathbb{Z}}$ the value taken by the environment at time $t$.
This is an element in the space $S^{\mathbb{Z}}$ which we endow with the product topology.
We also denote $\eta=(\eta_t)_{t \in \mathbb{R}_+}$, which will be called the random environment.

\begin{assumption}\label{hyp:invariantenv}We assume that the trajectories $t \mapsto \eta_t$ belong to $\mathcal{D}(\mathbb{R}_+, S^\mathbb{Z})$,  the space of all c\`adl\`ag functions from $\mathbb{R}_+$ to $S^\mathbb{Z}$.
We also assume that the random environment $\eta$ is invariant with respect to translations by elements of $\mathbb{L}$:
\begin{display}
  \label{e:invariantenv}
   for every $(z, s) \in \mathbb{L}$, the two processes $\big( \eta_t(x) \big)_{(x, t) \in \mathbb{L}}$\\ and $\big( \eta_{s+t}(z + x) \big)_{(x, t) \in \mathbb{L}}$ have the same distribution.
\end{display}
\end{assumption}
Fixing $z = 0$ and varying $s$ over $\mathbb{R}_+$ in Assumption~\ref{hyp:invariantenv} implies that  $\eta$ is stationary in time.

A box is defined to be any subset of $\mathbb{R}^2$ of the type $[a,b) \times [c,d)$.
For such a box, we call $b-a$ and $d-c$ its horizontal and vertical side lengths, respectively.
Given two boxes $B_i := [a_i,b_i) \times [c_i,d_i)$, $i = 1, 2$, with $c_2 > d_1$ we define their time-distance $d(B_1,B_2):=c_2 - d_1$.

Let $P$ be the law of the environment and $\Cov$ the covariance with respect to $P$.
The main assumption that we impose on our random environment is that it satisfies the following decoupling hypothesis.

\nc{c:decouple}
\begin{definition}[Decoupling inequality]
  \label{d:envdecouple}
  For $\uc{c:decouple}, \alpha > 0$, we say that $P$ satisfies the decoupling inequality $\mathscr{D}_{\mathrm{env}}(\uc{c:decouple}, \alpha)$ if the following holds.
For every $r \geq 1$, every pair of boxes $B_1, B_2 \subseteq \mathbb{R}^2$ having both side lengths at most equal to $5r$ and time-distance $d(B_1,B_2)\geq r$ and for any pair of functions $f_1, f_2:\Omega \to [0, 1]$ satisfying
  \begin{equation}
    f_i \in \sigma \big( \eta_t(x), \text{ with $(x, t) \in B_i \cap \mathbb{L}$} \big), \text{ for $i = 1, 2$},
  \end{equation}
we have
  \begin{equation}
    \Cov(f_1, f_2) \leq \uc{c:decouple} r^{-\alpha}.
  \end{equation}
\end{definition}

In Section~\ref{s:applications} we will present several models that satisfy the above decoupling condition, including the supercritical contact process (Section \ref{ss:contact}), some kinds of independent renewal chains (Section \ref{ss:renewal}) and Markov processes with positive spectral gap (Section \ref{ss:spectral}).

\subsection{Random walker}\label{ss:randomwalker}

On top of the dynamic random environment $\eta$, we  define a continuous-time random walker in one-dimension whose evolution  depends locally on $\eta$.
In this section we define these evolution rules and give the main assumptions we require on the joint law of the environment and of the random walker.

Given a starting point $y = (x, s) \in \mathbb{L}$ and  $t \in \mathbb{R}_+$, we will represent by $Y^y_t \in \mathbb{L}$ the space-time position of the walker after time $t$ has elapsed.
Let $\pi_1$ and $\pi_2$ denote the canonical orthogonal projections of $\mathbb{R}^2$ onto the first and the second coordinates, respectively.
We write $X^y_t:=\pi_1(Y^y_t) \in \mathbb{Z}$ for the spatial position of the random walker at time $t$.
Notice that $\pi_2(Y^y_t) = \pi_2(y) + t = s+t$.
We will sometimes write $Y_t^o$ (resp.\ $X_t^o$) for the space-time (space) position of the random walk starting at $o:=(0,0)$. 

We impose that the random walk trajectories $t\mapsto Y^y_t$ belong almost surely to the space
\begin{equation}
  \mathcal{D}_{\rm n.n.}([0, \infty), \mathbb{L}) := \bigg\{
  \begin{aligned}
    & \gamma:[0, \infty) \to \mathbb{L} \text{ c\`adl\`ag} \colon
     \big| \pi_1 \big( \gamma(t) \big) - \pi_1 \big( \gamma(t_-) \big) \big| \leq 1\\
    & \text{and }  \pi_2 \big( \gamma(t+s) \big) - \pi_2 \big( \gamma(t) \big) = s \text{ for $t,s  \in [0, \infty)$}
  \end{aligned}
  \bigg\},
\end{equation}
where $\gamma(t_-):=\lim_{s\nearrow t} \gamma(s)$.
In particular, this implies that the random walk performs only nearest neighbor jumps almost surely.
For every pair $0<T'<T''$ we define the set of paths $ \mathcal{D}_{\rm n.n.}([T', T''], \mathbb{L})$ from time $T'$ to $T''$ in an analogous way.
When $\gamma$ is an element in $\mathcal{D}_{\rm n.n.}([T',T''],\mathbb{L})$ we say that $\gamma$ has length $T''-T'$.

Let us now define the evolution of the random walker.
We start by introducing its allowed jumping times.
For each $x\in\mathbb{Z}$, let $(T^x_i)_{i = 1}^{\infty}$ be a random increasing collection of positive real numbers such that 
\begin{display}
$\{T_i^x,i=1,\ldots,\infty\}$ and $\{T_i^{x+1},i=1,\ldots,\infty\}$ are disjoint for all $x\in\mathbb{Z}$.
\end{display}  
For instance, one can keep in mind the example where the $(T^x_i)_{i = 1}^{\infty}$ are independent Poisson processes which are also (mutually) independent from the environment $\eta$.
We chose to work in a greater generality in order to include interesting applications.

The pairs $(x,T^x_i)_{x,i}$ will mark the space-time locations at which the random walker will be allowed to jump. 
This is encoded in the following definition.
\begin{definition}
  \label{d:allowed}
  Given a collection of jump times $(T_i^x)_{x,i}$, we say that an element $\gamma \in \mathcal{D}_{\rm n.n.}([0, \infty), \mathbb{L})$ is an \emph{allowed path} if all of its discontinuities are located at space-time points of the type $(x,T^x_i)_{x\in \mathbb{Z}}$.
  More precisely,
  \begin{display}
    if for some $t \in [0, \infty)$, $\gamma(t) = (x, s)$, then $\gamma({t + r}) = (x, s + r)$ for every $r < \underset{i}{\min} \{T^x_i - s \colon T^x_i > s\}$.
  \end{display}
  We define allowed paths in $\mathcal{D}_{\rm n.n.}([T'',T'], \mathbb{L})$ analogously.
\end{definition}

Besides the jump times $T_i^x$, we also fix independent uniform random variables $U^x_i \in [0, 1]$, also independent from all the rest, that provide the extra randomness that the random walker may use to determine its next jump.
As it will become clear below, this is done in order to encode the whole randomness of the walker, so that conditional on $\eta$, on the $T^x_i$'s and on the $U^x_i$'s, for each initial position, the walker follows a deterministic allowed path in $\mathcal{D}_{\rm n.n.}([0,\infty),\mathbb{L})$. 
For the rest of this paper we will denote by $\mathbb{P}$ the joint law of $\eta$, $(T_i^x)_{x,i}$ and $(U_i^x)_{x,i}$.

Let us now fix a positive integer $\ell$ and a function
\begin{equation}
\label{e:local_g}
g:S^{\{-\ell,\ldots,\ell\}} \times [0, 1] \to \{-1, 0, 1\},
\end{equation} which will be used to define the jumps of the random walker.
Roughly speaking, when the walker lies at site $x \in \mathbb{Z}$ and one of the arrival times $T^x_i$ comes up, the walker will jump to site $x + g(\eta_{T^x_i}(x-\ell),\ldots,\eta_{T^x_i}(x+\ell), U^x_i)$.

In a more precise way, we define the trajectory of the walker starting at $y$ to be the random element $(Y^y_t)_{t\in [0,\infty)}$ on $\mathcal{D}_{\rm n.n.}([0, \infty), \mathbb{L})$ which is completely determined by the following conditions:
\begin{enumerate} [\quad a)]
\item $Y^y_0 = y$ almost surely.
\item $(Y^y_t)_{t\in [0,\infty)}$ is an allowed path almost surely.
\item ``The jumps are determined by $g$''. That is,
  \begin{display}
    if $T^x_i = t$ and $Y^y_{t-} = (x, t)$, then $Y^y_{t} = \Big( x + g\big( \eta_t(x-\ell),\ldots, \eta_t(x+\ell), U^x_i \big), t \Big).$
  \end{display}
\end{enumerate}

The fact that the walker evolves in an one-dimensional environment and that it only performs nearest-neighbor jumps  together with the fact that the set of allowed jumping times for neighboring sites are disjoint almost surely implies a very important monotonicity property:
\begin{display}
  \label{e:monotone}
  if $x \leq x' \in \mathbb{Z}$ and $s \in \mathbb{R}_+$, then $X^{(x, s)}_t \leq X^{(x', s)}_t$ for every $t \geq 0$.
\end{display}
We are going to make strong use of this property for carrying on our proof.
This poses an obstacle for the task of extending our results for random walks with long-range jumps or in higher dimensions.

We now need to extend Assumption~\ref{hyp:invariantenv} and Definition~\ref{d:envdecouple} to the joint distribution $\mathbb{P}$ of the environment and the jump times $T_i^x$.

\begin{assumption}\label{hyp:invariant}
We assume that $\mathbb{P}$ is invariant with respect to translations by elements of $\mathbb{L}$:
\begin{display}
for every $(z, s) \in \mathbb{L}$, the two processes $\big(\big( \eta_t(x) \big)_{(x, t) \in \mathbb{L}},(T_i^x)_{x\in\mathbb{Z},i\geq 1}\big)$ and $\big(\big( \eta_{s+t}(z + x) \big)_{(x, t) \in \mathbb{L}},(T_i^{z+x}-s)_{x\in\mathbb{Z},i\geq 1\colon T_i^{z+x}>s}\big)$ have the same distribution.
\end{display}
\end{assumption}

\begin{definition}\label{d:decouple}
  For $\uc{c:decouple}, \alpha > 0$, we say that $\mathbb{P}$ satisfies the decoupling inequality $\mathscr{D}(\uc{c:decouple}, \alpha)$ if the following holds.
  For every $r \geq 1$, every pair of boxes $B_1, B_2 \subseteq \mathbb{R}^2$ having both side lengths at most equal to $5r$ and time-distance $d(B_1,B_2)\geq r$ and for any pair of functions $f_1, f_2:\Omega \to [0, 1]$ satisfying
  \begin{equation}
    f_i \in \sigma \big(\{ \eta_t(x);  (x, t) \in B_i \cap \mathbb{L}\}\cup\{(x, T_i^x); (x,T_i^x)\in B_i\cap \mathbb{L}\}\big), \text{ for $i = 1, 2$},
  \end{equation}
  we have
  \begin{equation}
    \mathbb{C}\mathrm{ov}(f_1,f_2) \leq \uc{c:decouple} r^{-\alpha}.
  \end{equation}
\end{definition}
Here $\mathbb{C}\mathrm{ov}$ stands for the covariance with respect to $\mathbb{P}$.
We also need a priori bounds on the speed of the random walker. 
For $v\in\mathbb{R}$, let
\begin{align}
A_T(v)=&\bigg\{
\begin{aligned}
 \text{there }& \text{exists $\gamma$ allowed path in $\mathcal{D}_{\rm n.n.}([0,T], \mathbb{L})$ such that}\\
&\text{$\gamma(0)\in[0,T)\times \{0\}$ and $\gamma(T)-\gamma(0)\geq vT$}
\end{aligned}
\bigg\},\\
\tilde{A}_T(v)=&\bigg\{
\begin{aligned}
\text{there }& \text{exists $\gamma$ allowed path in $\mathcal{D}_{\rm n.n.}([0,T], \mathbb{L})$ such that}\\
&\text{$\gamma(0)\in[0,T)\times \{0\}$ and $\gamma(T)-\gamma(0)\leq vT$}
\end{aligned}
\bigg\}.
\end{align}

\begin{assumption}\label{hyp:v+v-}
We assume that, for all $v>1$,
\begin{align}
\liminf_{T\to\infty}\mathbb{P}(A_T(v))=0
\,\,\,\text{ and }\,\,\, \liminf_{T\to\infty}\mathbb{P}(\tilde{A}_T(-v))=0.
\end{align}
\end{assumption}

We also need quantitative bounds on the probability of larger deviations above the maximum speed. Define
\begin{equation}
  F_{T} = \bigg\{
  \begin{array}{c}
    \text{$\exists$ allowed path $\gamma\in \mathcal{D}_{\rm n.n.}([0,T], \mathbb{L})$ with $\gamma(0)=0$ and a time}\\
    \text{$s \in [0,T]$ such that $[\gamma(s)-\ell,\gamma(s)+\ell]\nsubseteq [-2T,2T]\times[0,T]$}
  \end{array}
  \bigg\}.
\end{equation}
\begin{assumption}\label{hyp:FT}
There exists $c>0$ such that
\begin{equation}\label{e:FT}
\mathbb{P}(F_T)\leq c^{-1}e^{-cT}.
\end{equation}
\end{assumption}

\begin{remark}\label{rem:Poissontimes}
Note that Assumptions \ref{hyp:v+v-} and \ref{hyp:FT} should follow easily in most cases from a simple large deviations bound.
For instance, they are satisfied when the $(T_i^x)_{i\geq 1}$ are i.i.d.\ Poisson point processes of intensity $1$. 
If additionally they are independent of the environment and the environment law $P$ satisfies Assumption~\ref{hyp:invariantenv} (resp.\ $\mathscr{D}_{\mathrm{env}}(\uc{c:decouple},\alpha)$), then $\mathbb{P}$ satisfies Assumption~\ref{hyp:invariant} (resp.\ $\mathscr{D}(\uc{c:decouple},\alpha)$). 
\end{remark}

\subsection{Main theorems}
\label{s:main}

Our main result is the following law of large numbers and deviation bound for the random walker:

\begin{theorem}
  \label{t:main}
  Suppose Assumptions~\ref{hyp:invariant}, \ref{hyp:v+v-} and \ref{hyp:FT} are satisfied, as well as the decoupling property $\mathscr{D}(\uc{c:decouple}, \alpha)$ for some $\alpha > 8$, then there exists $v \in [-1,1]$ such that
  \begin{equation}
   \label{e:lln}
    \lim_{t \to \infty} \frac{X^o_t}{t} = v \qquad \mathbb{P}-\text{ a.s.}
  \end{equation}
  Moreover, for every $\epsilon >0$,
  \begin{equation}
   \label{e:ldb}
    \mathbb{P} \Big[ \Big| \frac{X^o_t}{t} - v \Big| \geq \epsilon \Big] \leq t^{-\alpha/4},
  \end{equation}
  for every $t$ large enough, depending on $\epsilon$.
\end{theorem}

The next theorem gives some conditions under which we can assure that the speed of the random walker is strictly positive.
This can be useful in several contexts as for instance when we study the distinguished zero of the East model in Subsection~\ref{ss:east}.

\begin{theorem}
  \label{t:positive}
 Suppose Assumptions~\ref{hyp:invariant}, \ref{hyp:v+v-} and \ref{hyp:FT} are satisfied, as well as the decoupling property $\mathscr{D}(\uc{c:decouple}, \alpha)$ for some $\alpha > 8.5$.
  Assume also that $g(\eta_{-\ell},\ldots,\eta_\ell, u) \in \{0, 1\}$ for every $(\eta_{-\ell},\ldots,\eta_\ell, u)\in S^{\{-\ell,\ldots,\ell\}}\times [0,1]$ and that
  \begin{equation}
    \label{e:one_jump}
    \mathbb{P} \big[ X^{o}_1 \geq 1 \big] > 0,
  \end{equation}
 Then both conclusions of Theorem~\ref{t:main} hold but in addition we conclude that the speed $v$ is strictly positive.
\end{theorem}

In other words, the above theorem gives that, if ``the random walker can only jump to the right'' and that ``starting at the origin, it has a positive probability of jumping within one time unit'' then $v>0$.
Clearly, an equivalent result yielding strictly negative speed holds when we only allow the random walk to jump to the left.

\begin{remark}
 The rate of decay in our deviation bound \eqref{e:ldb} is not optimal.
 It only reflects particularities of our renormalization techniques.
\end{remark}
\begin{remark}
 The lower bound we need to impose on $\alpha$ for Theorems \ref{t:main} and \ref{t:positive} is also not believed to be optimal.
 However, it is important to notice that the above result does not hold true if we weaken too much our decoupling condition.
In particular, we provide an example in Section~\ref{s:example} of a space-time environment satisfying a weaker decoupling hypothesis (in particular it is space-time ergodic) along with a natural random walker defined on it that fails to satisfy the Law of Large Numbers.
\end{remark}

The rest of the paper is organized as follows: Sections~\ref{s:strategy}--\ref{s:v_+=v_-} are devoted to the proof of Theorem~\ref{t:main}. In Section~\ref{s:positive} we prove Theorem~\ref{t:positive}. In Section~\ref{s:applications} we list a number of applications of our results and finally in Section~\ref{s:example} we provide a counter-example of a random walk on an ergodic environment that does not satisfy the LLN.

\section{Strategy of the proof}\label{s:strategy}

In this section, we give an overview of the idea behind the proof of Theorem~\ref{t:main} and define some important objects that will be used in the reminder of  the paper.

The main of these objects consist of two limiting values for the long-term speed of the random walker: the upper speed $v_+$ and lower speed $v_-$.
These quantities will play a central role in our arguments.
As we are going to prove below, their values coincide and are equal to the speed $v$ appearing in the statement of Theorem \ref{t:main}.

In order to define $v_+$ and $v_-$ precisely, let us first introduce an event whose occurrence indicates that the random walker has moved with average speed larger than $v\in\mathbb{R}$ during a certain interval of time.
For $H\in\mathbb{R}_+$ and $w\in\mathbb{R}^2$, we define
\begin{equation}
  \label{e:A_m_v}
  A_{H,w}(v) := \Big[ \text{there exists $y \in \big(w+[0,H)\times\{0\}\big) \cap \mathbb{L}$ s.t. $X^y_{H} - \pi_1(y) \geq v H$} \Big].
\end{equation}
See Figure~\ref{f:event_A} for an illustration of these events.

We want to study the probability of the events as in \eqref{e:A_m_v}.
In order to have a quantity that does not depend on the reference point we maximize in $w$, that is, we define
\begin{equation}
  \label{e:p_k}
  p_H(v) := \sup_{w \in \mathbb{R}^2} \mathbb{P} \big(A_{H,w}(v) \big) = \sup_{w \in [0, 1) \times \{0\}} \mathbb{P} \big(A_{H,w}(v) \big),
\end{equation}
where the second equality follows from stationarity and translation invariance.
Note that  $\big(w+[0,H)\times\{0\}\big) \cap \mathbb{L}$ takes at most two different values for $w$ varying in $[0,1)\times\{0\}$ so, in fact, the second supremum is taken over a finite set.
It is just meant to take into account cases where the reference point $w$ may not belong to $\mathbb{L}$.

\begin{figure}
  \centering
    \begin{tikzpicture}[use Hobby shortcut]
      \draw (0, 0) rectangle (10, 2);
      \draw[thick] (4, 0) -- (6, 0);
      \draw[thick] (4, -.15) -- (4, .15);
      \draw[thick] (6, -.15) -- (6, .15);
      \draw (4.5, 0) .. (4.8, .4) .. (5.5, .8) .. (5.5, 1.2) .. (6, 1.6) .. (6.5, 2);
      \draw[dashed] (4.5, 0) -- (6.2, 2);
      \draw[fill, below right] (4.5, 0) circle (.05) node {$y = Y^y_0$};
      \draw[fill, above right] (6.5, 2) circle (.05) node {$Y^y_{H}$};
      \draw[fill, above left] (6.2, 2) circle (.05) node {$y + H (v, 1)$};
      \draw[right] (10, 1) node {$H$};
      \draw [decorate,decoration={brace,amplitude=10pt}] (10, -.5) -- (0, -.5) node [black, midway, yshift=-0.6cm] {$5 H$};
    \end{tikzpicture}
  \caption{An illustration of the event $A_{H, o}(v)$.
  Starting from the point $y \in \big([0,H)\times\{0\}\big) \cap \mathbb{L}$ the walker attains an average speed larger than $v$ during the time interval $[0,H]$.}
  \label{f:event_A}
\end{figure}
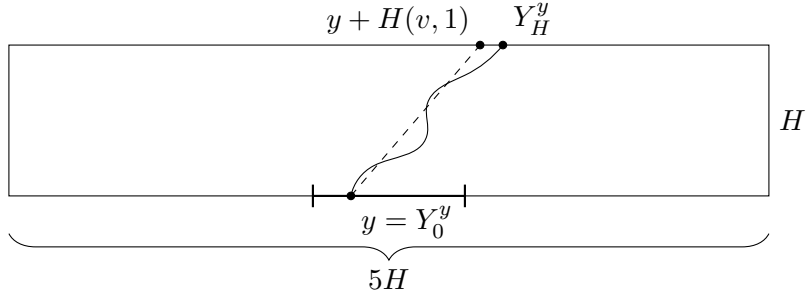

We can now introduce the \emph{upper speed} of the random walker as
  \begin{equation}
    \label{e:v_+}
    v_+ := \inf \big\{ v \in \mathbb{R} \colon \liminf_{H \to \infty} p_H (v) = 0 \big\}.
  \end{equation}

Similarly, we define its \emph{lower speed},
\begin{equation}
  \label{e:v_-}
  v_- := \sup \big\{ v \in \mathbb{R} \colon \liminf_{H \to \infty} \tilde{p}_H(v) = 0 \big\},
\end{equation}
where, analogously to the quantities $p_H(v)$ previously defined, we write
\begin{equation}\label{e:tildep_H}
  \tilde{p}_H(v) := \sup_{w \in \mathbb{R}^2} \mathbb{P}\big(\tilde{A}_{H,w}(v) \big)=  \sup_{w \in [0, 1) \times \{0\}} \mathbb{P}\big(\tilde{A}_{H,w}(v) \big)
\end{equation}
with
\begin{equation}
  \label{e:tilde_A}
  \tilde{A}_{H,w}(v):=\Big[\text{there exists $y \in \big(w+[0,H)\times\{0\}\big) \cap \mathbb{L}$ with $X^y_{H} - \pi_1(y) \leq v H$}\Big].
\end{equation}

\begin{remark}Assumption~\ref{hyp:v+v-} implies that $v_+ \leq 1$ and $v_- \geq -1$.
\end{remark}

Roughly speaking, the definition of $v_+$ implies that for any $v>v_+$, the probability that the average speed of the walker exceeds $v$ vanishes as the amount of time elapsed increases.
The next lemma shows that, it vanishes at least polynomially fast, provided that $\alpha$ is large enough.

\nc{c:deviation}
\begin{lemma}
  \label{l:deviation_interp}
  Suppose Assumptions~\ref{hyp:invariant}, \ref{hyp:v+v-} and \ref{hyp:FT} are satisfied, as well as the decoupling property $\mathscr{D}(\uc{c:decouple}, \alpha)$ with $\alpha > 5$, then for any $\epsilon > 0$ there exists a constant $\uc{c:deviation}=\uc{c:deviation}(\epsilon)$ such that
  \begin{equation}
  \label{e:deviation_interp}
    \begin{split}
      & p_H(v_+ + \epsilon) \leq \uc{c:deviation}H^{-\alpha/4} \text{ and}\\
      & \tilde{p}_H (v_- - \epsilon)\leq \uc{c:deviation}H^{-\alpha/4},
    \end{split}
  \end{equation}
  for every $H\in\mathbb{R}_+$.
\end{lemma}

This shows that $v_+$ and $v_-$ limit the rate of displacement of the random walk in the sense that the probability that it moves faster than $v_+$ or slower than $v_-$ decays fast.

It might be possible to conclude directly from their definitions that $v_+ \geq v_-$, but this is also a simple consequence of \eqref{e:deviation_interp}.

Indeed, assume by contradiction that $v_+ < v_-$ and define $\epsilon := (v_- - v_+)/2 >0$ and $\bar{v} :=(v_+ + v_-)/2$.
Equation \eqref{e:deviation_interp} implies that, for some $H$ large enough (depending on $\alpha$, $v_+$ and $v_-$) $\mathbb{P}(X^o_H \geq \bar{v}H) \leq 1/4$ and $\mathbb{P}(X^ o_H \leq \bar{v}H) \leq 1/4$ hold simultaneously providing a contradiction.

Having  \eqref{e:deviation_interp} it remains to show that $v_+ = v_-$, which will ultimately imply the desired law of large numbers \eqref{e:lln} and concentration estimate \eqref{e:ldb}.

\begin{lemma}\label{l:v_+=v_-}
 Suppose Assumptions~\ref{hyp:invariant}, \ref{hyp:v+v-} and \ref{hyp:FT} are satisfied, as well as the decoupling property $\mathscr{D}(\uc{c:decouple}, \alpha)$ with $\alpha > 8$, then
\begin{equation}
v_+=v_-.
\end{equation}
\end{lemma}

The intuition behind the proof of this lemma is fairly simple.
The implicit definition of $v_+$ and $v_-$ assures that the random walker has a good chance of attaining speeds close to both of these values over sufficiently long time scales.
On the other hand the probability that it runs faster than $v_+$ (and also slower than $v_-$) vanishes fast.
Assume by contradiction that $v_+ > v_-$.
By the fact that the random walker cannot run faster than $v_+$, the moments when its speed stays close to $v_-$ should delay it sufficiently to prevent it from attaining a speed close to $v_+$ over long time scales.
This would give rise to a contradiction implying that $v_{-} = v_{+}$.

Lemma~\ref{l:deviation_interp} will be proved in Section~\ref{s:deviations} and Lemma~\ref{l:v_+=v_-} in Section~\ref{s:v_+=v_-}.
Together they establish Theorem~\ref{t:main}.

\section{Upper and lower deviations of the speed}\label{s:deviations}

This section is devoted to the proof of Lemma  \ref{l:deviation_interp} via a renormalization procedure.
We only prove the decay of $p_H(v_++\epsilon)$; that of $\tilde{p}_H(v_--\epsilon)$ is completely analogous as it can be seen by considering the random walker obtained by replacing $g$ by $-g$.
Indeed, the upper speed for the new walker equals the negative of the lower speed for the original one.

The section is divided into three main parts.
In Section \ref{s:scales}, we establish the sequences of scales along which we analyze the system.
Next, in Section \ref{s:decay_scale}, we prove a version of Lemma \ref{l:deviation_interp} obtaining a power-law upper bound for $p_H$ similar to the one in \eqref{e:deviation_interp} but only for $H$ restricted to multiples of this sequence of scales.
In Section \ref{s:proof_deviation_interp} we interpolate in order to lift the restriction in the values of $H$.

\subsection{Scales and boxes}
\label{s:scales}
We start by defining recursively the following sequence
\begin{equation}
  \label{e:L_k}
  L_0 := 10^{10} \quad \text{and} \quad L_{k + 1} := l_k L_k \text{ for $k \geq 0$, where $l_k := \lfloor L_k^{1/4} \rfloor$}.
\end{equation}
These numbers will be used throughout the text in order to define the scales of time and space in which we analyze the displacement of the random walker.

\nc{c:scale_round}
Observe that there exists a constant $\uc{c:scale_round} > 0$ such that
\begin{equation}
  \label{e:scale_round}
  \uc{c:scale_round} L_k^{5/4} \leq L_{k + 1} \leq L_k^{5/4}, \text{ for every $k \geq 0$}.
\end{equation}
For a given integer-valued $L \geq 1$ and a real-valued $h \geq 1$, we define the box
\begin{equation}
  \label{e:B_L_h}
  B^h_L := [-2 h L, 3 h L) \times [0, h L) \subseteq \mathbb{R}^2,
\end{equation}
as well as the interval
\begin{equation}
  \label{e:I_L_h}
  I^h_L := [0, h L) \times \{0\} \subseteq \mathbb{R}^2,
\end{equation}
(see Figure~\ref{f:event_A} where $H=hL$).
In addition, for $w \in \mathbb{R}^2$, we denote
\begin{equation}
  \begin{split}
    B^h_L(w) & := w + B^h_L \,\, \text{ and }\\
    I^h_L(w) & := w + I^h_L.
  \end{split}
\end{equation}

\begin{remark}
Since the definitions of $B^h_L$ and $I^h_L$ depend only on the product $hL$, it may not be clear yet why we consider the double index.
It will in fact be very useful for us to use renormalization techniques in two steps, varying one parameter after the other. One can think of $h$ as a zooming parameter that, when increased, maps the discrete lattice into the continuous space.
Differently, $L$ (which will be chosen as $L_k$ later) can be thought of as the macroscopic size of a box.
\end{remark}

\begin{remark}
It is important to notice that $B^h_L(w)$ is a continuous box, meaning that it is defined as a subset of $\mathbb{R}^2$ rather than only of $\mathbb{L}$.
This choice will be useful and simplify the notation later when we will need to consider translations of these boxes by vectors of type $(vt, t)$ (for a given speed $v \in \mathbb{R}$) which are not necessarily elements of $\mathbb{L}$.
\end{remark}

In order to index the boxes and intervals defined   above in a more concise manner, we introduce the set of indices
\begin{equation}
  \label{e:M_h_k}
  M^h_k := \{h\} \times \{k\} \times \mathbb{R}^2,
\end{equation}
so that, for $m = (h,k,w) \in M^h_k$ and $v\in\mathbb{R}$, we can write
\begin{equation}
B_m := B^h_{L_k}(w)\,,\quad I_m := I^h_{L_k}(w)\quad \text{and}\quad A_m(v):=A_{hL_k,w}(v).
\end{equation}

For some of our purposes we need to assure that for $m\in M^h_k$, after starting at a point in $I_m \cap \mathbb{L}$, the random walker, remains inside $B_m$ up to time $hL_k$ (as well as the sites it needs to inspect to decide its jumps). This explains why we defined $B_m$ having its width bigger than its height.
For each $m=(h,k,w) \in M_k^h$ we define
\begin{equation}\label{e:boundedspeed}
  F_m := \Big[
  \begin{array}{c} \text{for every allowed path $\gamma$ starting at $I_m \cap \mathbb{L}$ with}\\  \text{length $h L_k$, $\left\{\gamma(\pi_2(w)+t)\colon t\in [0,hL_k]\right\}+[-\ell,\ell]\subseteq B_m$}
   \end{array}
   \Big].
 \end{equation}
From Assumption~\ref{hyp:FT}, we deduce easily that there exists\nc{c:F_m} $\uc{c:F_m} > 0$ such that
\begin{equation}
\label{e:F_m_decay}
\mathbb{P}(F^{\mathsf{c}}_m) \leq \uc{c:F_m}^{-1} \exp\{ -\uc{c:F_m} h L_k \}.
\end{equation}

\subsection{The decay of $p_H(v)$ along a particular sequence}
\label{s:decay_scale}
In this section we prove the following
\begin{lemma}\label{l:deviation_reduction}
Under the hypotheses of Lemma~\ref{l:deviation_interp}, for all $v>v_+$ there exists $\uc{c:h_hat} = \uc{c:h_hat}(v) \geq 1$ and $\uck{c:k_bar}=\uck{c:k_bar}(v) \geq 1$ such that for every $k \geq \uck{c:k_bar}$
\begin{equation}
\label{e:deviation_reduction}
p_{\uc{c:h_hat}L_k}(v) \leq L_k^{-\alpha/2}.
\end{equation}
\end{lemma}
The inequality \eqref{e:deviation_reduction} only concerns the decay of $p_H(v)$ for $H$ taking values along a specific sequence of multiples of the $L_k$'s (that depends on $v$).
We will prove it using a recursive inequality involving quantities that are close to $p_{hL_k}$ and $p_{hL_{k+1}}$ (Lemma~\ref{l:inductive}, \eqref{e:inductive}).
In turn, this recursive inequality follows from an intermediate result (Lemma~\ref{l:cascade}), which relates the occurrence of an event of the type $A_{hL_{k+1},w}$ with the occurrence of two events of the type $A_{hL_k,w'}$ supported on boxes that are well-separated in time.
As we have already mentioned, we will show in Section~\ref{s:proof_deviation_interp} how Lemma~\ref{l:deviation_interp} follows from the previous lemma via a simple interpolation procedure.

\nck{c:k_bar}
\nc{c:h_hat}

Let us start by introducing some few extra definitions.
Given $v >v_+$, fix an integer \nck{k:speed} $\uck{k:speed} = \uck{k:speed}(v)$ large enough so that
\begin{equation}
  \sum_{k \geq \uck{k:speed}} \frac{8}{l_k} < \frac{v - v_+}{2}.
\end{equation}
(recall the definition of $l_k$ below \eqref{e:L_k}).
Now set
\begin{equation}
 \label{e:v_k_o}
  v_{\uck{k:speed}} := \frac{v + v_+}{2} \quad \text{and} \quad v_{k + 1} := v_k + \frac{8}{l_k} \text{ for every $k \geq \uck{k:speed}$}.
\end{equation}
Note that $v_\infty := \lim_{k \to \infty} v_k < v_{\uck{k:speed}} + (v - v_+)/2 = v$.
In particular, $v_k \in (v_+, v)$ for every $k \geq \uck{k:speed}$ (see Figure \ref{fig:v_k}).
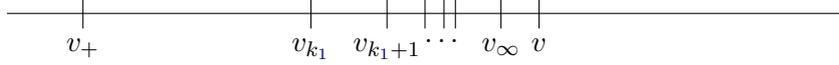
\begin{figure}
  \begin{center}
    \begin{tikzpicture}
      \draw (-1, 0) -- (10, 0);
      \draw (0, .2) -- (0, -.2) node [below] {$v_+$};
      \draw (3, .2) -- (3, -.2) node [below] {$v_{\uck{k:speed}}$};
      \draw (6, .2) -- (6, -.2) node [below] {$v$};
      \draw (4, .2) -- (4, -.2) node [below] {$v_{\uck{k:speed} + 1}$};
      \draw (4.5, .2) -- (4.5, -.2);
      \draw (4.9, .2) -- (4.9, -.2);
      \draw (4.75, .2) -- (4.75, -.2) node [below] {$\dots$};
      \draw (5.5, .2) -- (5.5, -.2) node [below] {$v_{\infty}$};
    \end{tikzpicture}
  \end{center}
  \caption{The sequence of velocities $v_k$ as defined in \eqref{e:v_k_o}}
  \label{fig:v_k}
\end{figure}

Given $m \in M^h_{k + 1}$ in the form $\big( h, k + 1, (z, t) \big)$, there exists a set $C_m \subseteq M^h_k$ satisfying
\begin{gather}
  \label{e:pave_number}
  |C_m| = 5 l_k^2 \text{ and}\\
  \label{e:pave_lines}
  \bigcup_{m' \in C_m} I_{m'} = B_m \cap \big( \mathbb{R} \times (t + h L_k \mathbb{Z}) \big).
\end{gather}

We are now ready to state a result that relates the occurrence of the events of the type $A_{hL_k,w}$ in two consecutive scales.
Recall that $d(B_1, B_2)$ stands for the time-distance between a pair of boxes as defined above Definition~\ref{d:decouple}.

\begin{lemma}
  \label{l:cascade}
  Let $k\geq k_1$. Given $m \in M^h_{k + 1}$, on the event $A_m(v_{k + 1}) \cap \big( \cap_{m' \in C_m} F_{m'} \big)$, there exist two indices $m_1, m_2 \in C_m$ such that
  \begin{equation}
  \label{e:dBm1Bm2}
    A_{m_i}(v_k) \text{ occurs for $i = 1, 2$ and }
    d(B_{m_1}, B_{m_2}) \geq h L_k.
  \end{equation}
\end{lemma}

\begin{proof}
First notice that $\cap_{m'\in C_m}F_{m'}\subset F_m$.
  Without loss of generality, we assume that $m$ is of the form $\big( h, k + 1, (z, 0) \big)$.
It is enough to show that
  \begin{display}
   \label{e:Amivk}
   there exist $m_1, m_2, m_3 \in C_m$ such that $A_{m_i}(v_k)$ occurs for each $i = 1, 2, 3$ and $(\pi_2(B_{m_i}))_{i = 1, 2, 3}$ are disjoint.
  \end{display}
Let us assume that \eqref{e:Amivk} does not hold, so that
  \begin{display}
    \label{e:three_fail}
    for all but at most two indices $j \in \{ 0, \dots, l_k - 1 \}$, $\big(A_{m'}(v_k)\big)^{\mathsf{c}}$ occurs for every box $B_{m'}$ with $m' \in C_m$ of type $m' = \big( h, k, (x, j h L_k) \big)$.
  \end{display}
For all $y \in I_m$, we can write
  \begin{equation}
    X^y_{hL_{k + 1}} - \pi_1(y)
    = \sum_{j = 0}^{l_k - 1} X^{Y^y_{j h L_k}}_{h L_k} - X^y_{j h L_k}.
  \end{equation}
  Note that, by \eqref{e:pave_lines} and the assumption that $F_m$ occurs, the points $Y^y_{j h L_k}$, $j=0,\ldots, l_k-1$, must belong to some $I_{m'}$ with $m' \in C_m$.
  When $A_{m'}$ does not occur, we can bound the corresponding difference in the right-hand side by $v_k h L_k$.
  Otherwise, we can use the occurrence of $F_{m'}$ to bound this difference by $3 h L_k$.
  Thus,
  \begin{equation}
  \label{e:displacement_control}
    \begin{split}
      X^y_{hL_{k + 1}} - \pi_1(y)
      & \overset{\eqref{e:three_fail}}\leq \big( l_k - 2 \big) v_k h L_k + 2 \cdot 3 h L_k\\
      & = v_k h L_{k + 1} + \Big( \frac{6 - 2v_k}{l_k} \Big) h L_{k + 1}\\
      & \overset{v_k > -1}< \Big( v_k + \frac{8}{l_k} \Big) h L_{k + 1} \overset{\eqref{e:v_k_o}}= v_{k + 1} h L_{k + 1}
    \end{split}
  \end{equation}
   which implies that $A_m$ does not occur.
\end{proof}

\begin{remark}
As shown in \eqref{e:displacement_control} the reason why we consider different speeds for each scale is to guarantee that: ``if the walker moves faster then $v_{k+1}$ inside a box at scale $L_{k+1}$ then it will move faster then $v_k$ inside well-separated boxes at scale $k$''.
This would not be necessarily true if we considered the same fixed speed for every scale.
\end{remark}

The previous step allows us to obtain an inductive inequality for the quantities $p_{hL_k}(v_k)$.

\nc{c:inductive}
\begin{lemma}
  \label{l:inductive}
  Suppose Assumptions~\ref{hyp:invariant}, \ref{hyp:v+v-} and \ref{hyp:FT} are satisfied, as well as the decoupling property $\mathscr{D}(\uc{c:decouple}, \alpha)$ with $\alpha > 5$. There exists $\uc{c:inductive} >0$  such that, given $v >v_+$,
    \begin{display}\label{e:hyprecp}
   if for some $h \geq 1$ and $k \geq \uck{k:speed} \vee \uc{c:inductive}$ we have
    $p_{hL_k}(v_k) \leq L_k^{-\alpha/2}$  then $p_{hL_{k + 1}}(v_{k + 1}) \leq L_{k + 1}^{-\alpha/2}$.
  \end{display}
\end{lemma}

Note that the constant $\uc{c:inductive}$ is uniform in $h\geq 1$.

\begin{proof}
  Fix $m = (h,k+1,w) \in M^h_{k+1}$ and let $(m_1,m_2)_m$ denote the set of all pairs of indices $m_1$, $m_2$ in $C_m$ whose corresponding boxes $B_{m_1}$ and $B_{m_2}$ are separated by a time-distance at least equal to $h L_k$.
We perform the following sequence of inequalities, whose steps are justified below:
  \begin{equation}
  \label{e:inductive}
   \begin{split}
    \mathbb{P}(A_m(v_{k+1}))  &\leq \mathbb{P} \Big[A_m(v_{k+1}) \cap (\cap_{m' \in C_m} F_{m'}) \Big] + \mathbb{P} \Big[ \cup_{m' \in C_m} F^{\mathsf{c}}_{m'} \Big] \\
    & \leq 25\, l_k^4 \sup_{(m_1, m_2)_m} \mathbb{P} [ A_{m_1} (v_k)\cap A_{m_2}(v_k) ]
    + 5 l_k^2 \sup_{m' \in C_m} \mathbb{P}[F_{m'}^{\mathsf{c}}] \\
    & \leq 25\, l_k^4 \big(p_{hL_k}(v_k)^2+\uc{c:decouple}(h L_k)^{-\alpha}\big) + 5 l_k^2 \uc{c:F_m}^{-1} e^{ -\uc{c:F_m} h L_k }\\
    & \leq 25\, l_k^4 \big(p_{hL_k}(v_k)^2+c(h L_k)^{-\alpha}\big).
   \end{split}
  \end{equation}
 To obtain the second inequality we used \eqref{e:pave_number} and Lemma~\ref{l:cascade}.
 For the third inequality we employed the hypothesis that $\mathscr{D}(\uc{c:decouple}, \alpha)$ is satisfied in order to decouple $A_{m_1}$ and $A_{m_2}$ which is possible since they are supported in boxes that are well-separated vertically.
 We also used \eqref{e:F_m_decay}.

 Now, taking the supremum over all $m \in M^h_{k+1}$ in the LHS of \eqref{e:inductive} and dividing by $L^{-\alpha/2}_{k+1}$ we get
  \begin{equation}
    \begin{split}
      \frac{p_{hL_{k + 1}}(v_{k + 1})}{L_{k + 1}^{-\alpha / 2}} & \leq 25 L_{k + 1}^{\alpha / 2}\, l_k^4 \; \big( p_{hL_k}(v_k)^2 + c (h L_k)^{-\alpha} \big)\\
      & \leq 25 L_k^{5\alpha / 8 + 1} \big( L_k^{-\alpha} + c L_k^{-\alpha} \big) \leq c L_k^{1 - 3 \alpha / 8}\overset{\alpha>5}{\leq}cL_k^{-7/8},
    \end{split}
  \end{equation}
where we used \eqref{e:scale_round}, \eqref{e:hyprecp} and $h\geq 1$ in the second inequality. The right-hand side is bounded by $1$ as soon as $k \geq \uc{c:inductive}$ for a sufficiently large $\uc{c:inductive}$ depending on  $\uc{c:decouple}, \uc{c:F_m}$.
\end{proof}

We are now ready to conclude the proof of Lemma~\ref{l:deviation_reduction}.
\begin{proof}[Proof of Lemma~\ref{l:deviation_reduction}]
Assume $v >v_+$ and fix $\uck{c:k_bar} = \uck{c:k_bar}(v):=\uck{k:speed}(v)\vee\uc{c:inductive}$.
 Since $v_{\uck{c:k_bar}} > v_+$, we have
  \begin{equation}
    \liminf_{h \to \infty} p_{hL_{\uck{c:k_bar}}}(v_{\uck{c:k_bar}}) = 0.
  \end{equation}
 Therefore, we can fix $\uc{c:h_hat}(v)\geq 1$ for which
  \begin{equation}
    p_{\uc{c:h_hat}L_{\uck{c:k_bar}}}(v_{\uck{c:k_bar}}) \leq L_{\uck{c:k_bar}}^{-\alpha / 2}.
  \end{equation}
Now we can use Lemma~\ref{l:inductive} to obtain recursively
  \begin{equation}
    p_{\uc{c:h_hat}L_k}(v) \leq p_{\uc{c:h_hat}L_k}(v_k) \leq L_k^{-\alpha / 2}
  \end{equation}
  for every $k \geq \uck{c:k_bar}$.
This proves Lemma~\ref{l:deviation_reduction}.

\end{proof}

\subsection{Proof of Lemma~\ref{l:deviation_interp}}
\label{s:proof_deviation_interp}

With Lemma~\ref{l:deviation_reduction} at hand, we just need an interpolation argument to establish Lemma~\ref{l:deviation_interp}. Let $v=v_++\epsilon$, $v' = (v_+ + v)/2$ and let $\uc{c:h_hat}(v')$ and $\uck{c:k_bar}(v')$ be as in Lemma~\ref{l:deviation_reduction}. For $H \in \mathbb{Z}_+$ let us define $\bar{k}$ as being the integer that satisfies:
\begin{equation}
 \label{e:change_scale}
  (\uc{c:h_hat} L_{\bar{k}})^{11/10} \leq H < (\uc{c:h_hat} L_{\bar{k} + 1})^{11/10}.
\end{equation}
Above, the choice $11/10$ for the exponent does not play an important role and it could be replaced by any number bigger than one.
Let us first assume that $H$ is sufficiently large so that $\bar{k} \geq \uck{c:k_bar}$ (which is possible, by \eqref{e:change_scale}).
Therefore, we can apply Lemma~\ref{l:deviation_reduction} to conclude that
\begin{equation}\label{e:deviation_reduction}
p_{\uc{c:h_hat}L_{\bar{k}}}(v')\leq L_{\bar{k}}^{-\alpha/2}.
\end{equation}

Now, in order to bound $p_H(v)$, we are going to start by fixing some $w \in [0, 1) \times \{0\}$ and pave the box $B^1_H( w)$ with boxes $B_m$ with $m \in M_{\bar{k}}^{\uc{c:h_hat}}$.
The set of indices of boxes used for such a paving is
\begin{equation}
  M = \big\{ m = (\uc{c:h_hat}, \bar{k}, \bar{w}) \in M^{\uc{c:h_hat}}_{\bar{k}} \colon \bar{w} \in\uc{c:h_hat} L_{\bar{k}}\,\mathbb{Z}^2 \text{ and } B_m \cap B^1_H( w) \neq \varnothing \big\},
\end{equation}
which satisfies
\begin{equation}\label{e:cardinalM}
  |M| \leq 6 \Big( \frac{H}{\uc{c:h_hat} L_{\bar{k}}} \Big)^2
  \overset{\eqref{e:change_scale}}\leq 6 \Big( \frac{(\uc{c:h_hat} L_{\bar{k} + 1})^{11/10}}{\uc{c:h_hat} L_{\bar{k}}} \Big)^2 \overset{\eqref{e:scale_round}}\leq c(v) L_{\bar{k}}^{3/4}.
\end{equation}

An important observation at this point is that, on the event $\cap_{m\in M} (A_m(v'))^{\mathsf{c}}$, for any $y \in I^1_H( w)$ the displacement of $X^y$ up to time ${\lfloor H/\uc{c:h_hat}  L_{\bar{k}} \rfloor \uc{c:h_hat} L_{\bar{k}}}$ can be bounded by
\begin{equation}
 \begin{split}
  X^y_{\lfloor H/\uc{c:h_hat} L_{\bar{k}} \rfloor \uc{c:h_hat} L_{\bar{k}}} - \pi_1(y)
  & = \sum_{j = 0}^{\lfloor H/\uc{c:h_hat} L_{\bar{k}} \rfloor - 1}
  X^{Y^y_{j \uc{c:h_hat} L_{\bar{k}}}}_{\uc{c:h_hat} L_{\bar{k}}} - X^y_{j \uc{c:h_hat} L_{\bar{k}}} \\
  & \leq v' \lfloor H/\uc{c:h_hat} L_{\bar{k}} \rfloor \uc{c:h_hat} L_{\bar{k}} \leq v' H.
 \end{split}
\end{equation}
where we used that $A_m(v')$ does not occur for any $m \in M$ and that each point $X^y_{j \uc{c:h_hat} L_{\bar{k}}}$ belongs to $I_m \cap \mathbb{L}$ for some $m \in M$.

Note that ${\lfloor H/\uc{c:h_hat}  L_{\bar{k}} \rfloor \uc{c:h_hat} L_{\bar{k}}}$ is approximately equal to $H$, but not exactly.
Therefore, we still need to bound the probability that the random walk has a large displacement between times $\lfloor H/\uc{c:h_hat} L_{\bar{k}} \rfloor \uc{c:h_hat} L_{\bar{k}}$ and $H$.
But, in fact, Assumption~\ref{hyp:FT} shows that for any $y\in\mathbb{L}$  
  \begin{equation*}
  \begin{split}
    \mathbb{P} & \big[ X^y_{H} -  X^y_{\lfloor H/\uc{c:h_hat} L_{\bar{k}} \rfloor \uc{c:h_hat} L_{\bar{k}}} \geq (v_+ - v') H \big] \\
  & \leq 4H\,\mathbb{P} \Big[ \exists\ \text{ allowed path $\gamma \in \mathcal{D}_{\rm n.n.}\big([0,H-\lfloor H/\uc{c:h_hat}  L_{\bar{k}} \rfloor \uc{c:h_hat} L_{\bar{k}}],\mathbb{L}\big)$ s.t.\ $\gamma(0)=0$ and}\\
&\ \gamma\big(H-\lfloor H/\uc{c:h_hat}  L_{\bar{k}} \rfloor \uc{c:h_hat} L_{\bar{k}}\big)\geq (v_+-v')H\Big]
+ c^{-1}e^{-c\lfloor H/\uc{c:h_hat}  L_{\bar{k}} \rfloor \uc{c:h_hat} L_{\bar{k}}}\\
& \overset{\eqref{e:FT}}\leq c(v)^{-1}He^{-c(v)L_{\bar{k}}},
  \end{split}
  \end{equation*}
as soon as $ (v_+-v')H\geq 2\uc{c:h_hat}L_{\bar{k}}$. Above, in the first inequality we used Assumption~\ref{hyp:FT} to find a union bound on the possible positions of  $X^y_{\lfloor H/\uc{c:h_hat}  L_{\bar{k}} \rfloor \uc{c:h_hat} L_{\bar{k}}}$, then translation invariance of $\mathbb{P}$.

Joining the two last estimates, we get for large enough $H$
\begin{equation}
  \begin{split}
    \mathbb{P} \big(A_{H,w}(v)\big)
    & = \mathbb{P} \big[X^y_{H} - y\geq v H\text{ for some } y \in I^1_H(w) \cap \mathbb{L} \big]\\
    & \leq \mathbb{P} [A_m(v') \text{ occurs for some $m \in M$}] + c^{-1}H^2 \exp\{-c L_{\bar{k}}\}\\
    & \overset{\eqref{e:deviation_reduction},\eqref{e:cardinalM}}\leq c L_{\bar{k}}^{3/4} L_{\bar{k}}^{-\alpha / 2} + c^{-1}\exp\{-cL_{\bar{k}}\}\\
    & \overset{\alpha > 5}\leq c L_{\bar{k}}^{-7 \alpha / 20} \overset{\eqref{e:change_scale}}{\leq} c H^{-\alpha/4}.
  \end{split}
\end{equation}
The conclusion of Lemma~\ref{l:deviation_interp} now follows by taking the supremum over all $w \in [0,1) \times \{0\}$ and then properly choosing the constant $\uc{c:deviation}$ in order to accommodate small values of $H$.

\section{Threats on the upper speed}
\label{s:threat_points}

As we discussed above, we want to show that $v_+ = v_-$ arguing by contradiction: If $v_+>v_-$, then spending a significant proportion of its time moving with speed close to $v_-$ will prevent the random walker to attain an average speed close to $v_+$ over long interval of times.
This contradicts the very definition of $v_+$.
The main goal of the present section is to prove preliminary results that will be used to formalize this argument in the next section.

Let us define
\begin{equation}
  \label{e:delta}
\delta:=\frac{v_+ - v_-}{4}.
\end{equation}
Note $\delta \in (0, 1/2]$, since we argue by contradiction and assume that $v_+ > v_-$.

\subsection{Trapped points}

\begin{definition}
  Given $H \geq 1$ and $\delta$ as in \eqref{e:delta}, we say that a point $w \in \mathbb{R}^2$ is $H$-trapped if there exists some $y \in \big(w + [\delta H, 2 \delta H] \times \{0\} \big) \cap \mathbb{L}$ such that
  \begin{equation}
  \label{e:trapped_point}
    X^y_H - \pi_1(y) \leq (v_- + \delta)H.
  \end{equation}
  Note that this definition applies to points $w \in \mathbb{R}^2$ that do not necessarily belong to $\mathbb{L}$.
\end{definition}

The key fact behind the above definition is the following: if $w$ is trapped, then starting from a nearby space-time point to the right of $w$, the random walker will be delayed in the near future (in the sense that its average speed will be bounded away from $v_+$).
In fact, according to \eqref{e:monotone}, if $w$ is $H$-trapped, then for every $w' \in \big( w + [0, \delta H] \times \{0\} \big) \cap \mathbb{L}$, we have
\begin{equation}
  \label{e:trapped_slow}
  X^{w'}_H - \pi_1(w') \leq X^y_H - \pi_1(y) + 2 \delta H \leq (v_- + 3 \delta) H = (v_+ - \delta) H,
\end{equation}
where $y$ is any point in $\big(w + [\delta H, 2 \delta H] \times \{0\} \big) \cap \mathbb{L}$ satisfying \eqref{e:trapped_point}.

The implicit definition of $v_-$ guarantees that a point is trapped with uniform positive probability in the following sense:

\nc{c:some_trapped}
\nc{c:H_lower}
\begin{lemma}
  \label{l:some_trapped}
  There exist constants $\uc{c:some_trapped}>0$ and $\uc{c:H_lower}>4/\delta$ (depending on the value of $\delta$ given  in \eqref{e:delta}), such that
  \begin{equation}
    \label{e:some_trapped}
    \inf_{H \geq \uc{c:H_lower}} \;\; \inf_{w \in [0, 1) \times \{0\}} \mathbb{P} \big[ \text{$w$ is $H$-trapped} \big] \geq \uc{c:some_trapped}.
  \end{equation}
\end{lemma}

\begin{proof}
Since $v_-+\delta>v_-$, the definition of $v_-$ implies that
\begin{equation}
\uc{c:some_trapped} := \frac{1}{2} \Big\lceil \frac{2}{\delta} \Big\rceil^{-1} \liminf_{H\to\infty} \tilde{p}_{H}(v_-+\delta)>0.
\end{equation}
In particular, there exists $\uc{c:H_lower}>4/\delta$ such that
\begin{equation}
\Big\lceil \frac{2}{\delta} \Big\rceil^{-1} \inf_{H\geq \uc{c:H_lower}} \tilde{p}_H(v_-+\delta) \geq \uc{c:some_trapped}.
\end{equation}
Then, for each $H > \uc{c:H_lower}$ we have
  \begin{equation*}
    \begin{split}
     \uc{c:some_trapped} & \leq \Big\lceil \frac{2}{\delta} \Big\rceil^{-1}
        \sup_{w \in [0, 1) \times \{0\}}
      \mathbb{P} \Big[
      \begin{array}{c}
        \text{there exists $y \in (w + [0, H) \times \{0\})\cap \mathbb{L}$}\\
         \text{such that $X^y_{H} - \pi_1(y) \leq (v_- + \delta) H$}
      \end{array}
      \Big]\\
      & \leq  \sup_{w \in [0, 1) \times \{0\}}
      \mathbb{P} \Big[
      \begin{array}{c}
        \text{there exists $y \in (w + [0, (\delta/2) H) \times \{0\})\cap \mathbb{L}$}\\
         \text{such that $X^y_{H} - \pi_1(y) \leq (v_- + \delta) H$}
      \end{array}
      \Big]\\
      & \leq \inf_{w \in [0, 1) \times \{0\}}
      \mathbb{P} \Big[
      \begin{array}{c}
        \text{there exists $y \in (w + [0, \delta H) \times \{0\})\cap \mathbb{L}$}\\
         \text{ such that $X^y_{H} - \pi_1(y) \leq (v_- + \delta) H$}
      \end{array}
      \Big] \\
      & = \inf_{w \in [0, 1) \times \{0\}}
      \mathbb{P} \Big[
      \begin{array}{c}
        \text{there exists $y \in (w + [\delta H, 2 \delta H) \times \{0\})\cap \mathbb{L}$}\\
         \text{ such that $X^y_{H} - \pi_1(y) \leq (v_- + \delta) H$}
      \end{array}
      \Big].
    \end{split}
  \end{equation*}

    In the second inequality we have split $[0, H)$ into intervals of length $\delta H/2$ and used a union bound.
  For the third inequality, we used translation invariance and the fact that $\delta H > 4$ (since $\uc{c:H_lower} > 4/\delta$) which implies that, for any $w\in[0,1)\times\{0\}$, the interval $w + [0, (\delta/2) H)\times \{0\}$ is contained in every interval $w'+[0, \delta H)\times \{0\}$ with $w'\in [-1,0)\times \{0\}$.
  Translation invariance was also used to obtain the last equality.

\end{proof}

The lemma above is a good step towards the proof that the walker will not be able to attain average speed close to $v_+$ over long time periods.
Indeed, one could think of the set of $H$-trapped points as a percolation-type environment of obstacles.
Every time the random walker passes next to such an obstacle it will be delayed up to time $H$.
Furthermore, if the probability that a point is trapped could be made very high, then every allowed path would have to approach these obstacles at time scales smaller than $H$ and we would be done.
However, Lemma \ref{l:some_trapped} only assures that this  probability is positive and it could, in principle, be very small.
Therefore, the random walker could always avoid these trapped points, or it could spend only a negligible fraction of the time next to them.

For this reason we introduce a more elaborate way of delaying the random walker.
Given a reference space-time point we look for the existence of at least one trap lying along a line segment with slope $v_+$ starting from this point (see Figure~\ref{f:threatened}).
If we are successful, the reference point is called a \emph{threatened point}.
As one would expect, the probability that a point is threatened becomes very high as we increase the length of the segment.
The key observation is that, if the random walk starts at a threatened point, it will most likely end up finishing to the left of the tip of the segment (see Figure \ref{f:threatened}), that is, it will be delayed with respect to $v_+$.
The details will be presented in the following section.

\subsection{Threatened points}

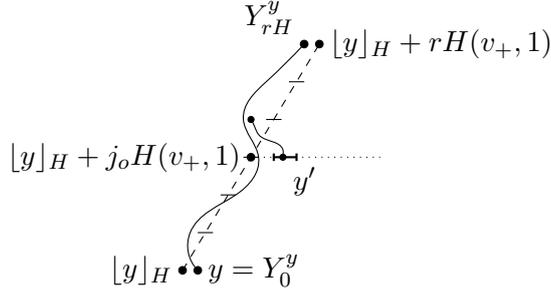
\begin{figure}
  \begin{center}
    \begin{tikzpicture}[use Hobby shortcut]
      \draw[dashed] (4.4, 0) -- (6.2, 3);
      \foreach \x in {1,...,5}
      { \draw[fill] (4.3 + \x / 3.3333, \x / 2) -- (4.5 + \x / 3.33, \x / 2);
      }
      \draw[fill, left] (5.3, 1.5) circle (.05) node {$\lfloor y \rfloor_{H} + j_o H(v_+, 1)$};
      \draw[dotted] (5.3, 1.5) -- (7, 1.5);
      \draw[thick] (5.6, 1.5) -- (5.9, 1.5);
      \draw[thick] (5.6, 1.45) -- (5.6, 1.55);
      \draw[thick] (5.9, 1.45) -- (5.9, 1.55);
      \draw[fill, left, right] (4.6, 0) circle (.05) node {$y = Y^y_0$};
      \draw[fill, left] (4.4, 0) circle (.05) node {$\lfloor y \rfloor_{H}$};
      \draw[fill, right] (6.2, 3) circle (.05) node {$\lfloor y \rfloor_{H} + rH (v_+, 1)$};
      \draw[fill, above left] (6, 3) circle (.05) node {$Y^y_{r H}$};
      \draw (5.7, 1.5) .. (5.72, 1.6) .. (5.6, 1.7) .. (5.4, 1.8) .. (5.35, 1.9) .. (5.3, 2);
      \draw[fill, below right] (5.72, 1.5) circle (.04) node {$y'$};
      \draw[fill, below right] (5.3, 2) circle (.04);
      \draw (4.6, 0) .. (4.5, .5) .. (5.1, 1) .. (5.4, 1.5) .. (5.2, 2) .. (5.5, 2.5) .. (6, 3);
    \end{tikzpicture}
    \caption{The point $\lfloor y \rfloor_{H}$ is $(H, r)$-threatened, since $\lfloor y \rfloor_{H} + j_o H (v_+, 1)$ is $H$-trapped.}
    \label{f:threatened}
  \end{center}
\end{figure}

\begin{definition}
  \label{d:threatened}
  Given $\delta$ as in \eqref{e:delta}, $H \geq 1$ and some integer $r \geq 1$, we say that a point $w \in \mathbb{L}$ is $(H, r)$-threatened if $w + j H (v_+, 1)$ is $H$-trapped for some $j = 0, \dots, r-1$.
\end{definition}

As we are going to show below, being on a threatened point will most likely impose a delay to the walker similarly to being next to trapped points.
Before proving this, we will introduce a notation for rounding of points in $\mathbb{L}$.

For $y = (x, t) \in \mathbb{L}$ and $H > 4 \delta^{-1}$, we define
\begin{equation}
\label{e:round_H}
  \big\lfloor y \big\rfloor_H = \Big( \Big\lfloor \frac{x}{\widetilde{H}} \Big\rfloor \widetilde{H} , t \Big), \text{ where $\widetilde{H} = {\lfloor \delta H/4 \rfloor}$},
\end{equation}
which is the closest point to the left of $y$ in the set $\widetilde{H} \mathbb{L}$ where the spatial coordinates are rescaled by $\lfloor \delta H/4 \rfloor$ .
Note that $\widetilde{H}$ is an integer so $\lfloor y \rfloor_H \in \mathbb{L}$.
Recall that the constant $\uc{c:H_lower}$ that was introduced in Lemma~\ref{l:some_trapped} was chosen in such a way that $\uc{c:H_lower} > 4 \delta^{-1}$, so that the rounding in \eqref{e:round_H} can be used for any fixed $H \geq \uc{c:H_lower}$.

Before we continue, let us briefly explain the reason why we introduce the above rounding.
In what follows, we will need to prove that there exist many threatened points within certain boxes.
However, in order to obtain a union bound that is uniform over $H$, we will only look for such points in a certain sub-lattice contained in $(\widetilde{H}\mathbb{Z})\times\mathbb{R}$.
This will become more clear in \eqref{e:all_are_threatened}, see also Remark~\ref{r:rounding}.

\medskip

Recall the definition of the scale sequence $(L_k)$, in \eqref{e:L_k}.
As promised above, the next (deterministic) lemma states that, once the walker gets next to a threatened point, either it runs faster than $v_+$ for a certain time interval (in order to overshoot the nearby trap) or else it will ultimately be delayed with respect to $v_+$.
See Figure \ref{f:threatened} for an illustration.

\begin{lemma}
  \label{l:delays}
  For any positive integer $r $ and any real number $H\geq \uc{c:H_lower}$, if we start the walker at some $y \in \mathbb{L}$ such that
  \begin{equation}
    \big\lfloor y \big\rfloor_{H} \text{ is $(H, r)$-threatened},
  \end{equation}
  then either
  \begin{enumerate}
  \item ``the walker runs faster than $v_+$ for some time interval of length $H$'', that is,
    \begin{equation}
      \label{e:speedup}
      X_{(j + 1)H}^y - X_{jH}^y
      \geq \Big( v_+ + \frac{\delta}{2 r} \Big)H \quad \text{ for some $j = 0, \dots, r - 1$,}
    \end{equation}
  \item or else, ``it will be delayed'', that is,
    \begin{equation}
      \label{e:delay}
      X^y_{rH} - \pi_1(y)
      \leq \Big( v_+ - \frac{\delta}{2 r} \Big) r H.
    \end{equation}
  \end{enumerate}
\end{lemma}

\begin{proof}
Fix $r \geq 1$ and $H\geq \uc{c:H_lower}$ as in the statement.
Assume that the point $\lfloor y \rfloor_{H}$ is $(H, r)$-threatened.
  Thus, for some $j_o \in \{ 0, \dots, r - 1\}$,  \begin{equation}
    \lfloor y \rfloor_{H} + j_o H (v_+, 1) \text{ is $H$-trapped}
  \end{equation}
  or, in other words, there exists a point
  \begin{equation}
    \label{e:location_y'}
    y' \in \Big( \big( \lfloor y \rfloor_{H} + j_o H (v_+, 1) \big) + [\delta H, 2 \delta H] \times \{ 0 \} \Big) \cap \mathbb{L}
  \end{equation}
  such that
  \begin{equation}
    \label{e:delay_y'}
    X^{y'}_{ H} - \pi_1(y') \leq (v_- + \delta) H= (v_+ - 3 \delta) H.
  \end{equation}
  Fix such a point $y'$ and notice from \eqref{e:location_y'} that,
  \begin{equation}\label{e:location_y'_2}
    \frac{3}{4} \delta H \leq
    \pi_1(y') - \big( \pi_1(y) + j_o Hv_+ \big)
    \leq 2 \delta H.
  \end{equation}

We now assume that \eqref{e:speedup} does not hold and bound the horizontal displacement of the random walk in three steps: before time $j_o H$, between times $j_o H$ and $(j_o + 1) H$ and from time $(j_o + 1) H$ to time $r H$.

  \begin{equation*}
    \begin{split}
      X^y_{j_oH} - \pi_1(y)
      & \leq \sum_{j = 0}^{j_o - 1} X^y_{(j+1) H} - X^y_{j H}\\
      & \overset{\neg\eqref{e:speedup}}\leq j_o \Big( v_+ + \frac{\delta}{2 r} \Big)H
     \leq j_o v_+H + \frac{\delta}{2 } H\\
      & \leq j_o v_+H+\frac{3}{4}\delta H.
    \end{split}
  \end{equation*}
  So, by \eqref{e:location_y'_2}, $Y^y_{j_o H}$ lies to the left of $y'$ and, by monotonicity, \eqref{e:delay_y'} and \eqref{e:location_y'_2} we have that
  \begin{equation}
    \begin{split}
      X^y_{(j_o + 1) H} & \leq X^{y'}_{H} \leq \pi_1(y') + (v_+ - 3 \delta) H\\
      & \leq \pi_1(y) + j_o v_+ H + 2 \delta H+ (v_+ - 3 \delta) H\\
      & \leq \pi_1(y) + (j_o + 1) v_+ H - \delta H.
    \end{split}
  \end{equation}
  Now applying once more the assumption that \eqref{e:speedup} does not hold, for $j = j_o, \ldots, r-1$, we can bound the overall displacement of the random walk up to time $rH$:
  \begin{equation}
    \begin{split}
      X^y_{r H} - \pi_1(y)
      & \leq \big( X_{r H}^y - X_{(j_o + 1) H}^y \big) + \big( X_{(j_o + 1) H}^y - \pi_1(y) \big)\\
      & \leq (r - j_o - 1) \Big( v_+ + \frac{\delta}{2r} \Big) H + (j_o + 1) v_+ H - \delta H\\
      & \leq r v_+ H - \frac{\delta}{2} H = \Big( v_+ - \frac{\delta}{2 r} \Big) r H,
    \end{split}
  \end{equation}
  showing that \eqref{e:delay} holds and thus proving the result.
\end{proof}

\subsection{Density of threatened points}
The next lemma is the main result of this section.
\nc{c:threatened}
\begin{lemma}[Threatened points]
  \label{l:threatened}
 Assume that $\mathbb{P}$ satisfies Assumptions~\ref{hyp:invariant}, \ref{hyp:v+v-} and \ref{hyp:FT}, as well as $\mathscr{D}(\uc{c:decouple},\alpha)$ for some $\alpha \geq 1$ and let $\delta$ be as defined in \eqref{e:delta}.
  There exists $\uc{c:threatened} = \uc{c:threatened}(\delta) > 0$ such that for every $H\geq\uc{c:H_lower}$
  \begin{equation}
    \mathbb{P} \big[ \text{$0$ is not $(H, r)$-threatened} \big] \leq \uc{c:threatened} r^{-\alpha},
  \end{equation}
  for any $r \geq 1$.
  Note that the above bound is uniform on $H \geq \uc{c:H_lower}$.
\end{lemma}

The conclusion of this lemma is useful because of the following:
despite the fact that it is conceivable that the trajectory of the random walker could avoid trapped points, the probability that a point is threatened can be made so high (by taking $r$ large) that with very high probability the random walker (and actually any path) cannot avoid spending a significant proportion of its time close to threatened points (as we will prove in the next section).
By Lemma \ref{l:delays}, if the random walker stands  close to a threatened point, then it has to run faster then $v_+$ for a certain interval of time in order to avoid being delayed.
However, by Lemma \ref{l:deviation_interp}, it is very unlike that it will be able to do it.
Thus with high probability, a delay with respect to $v_+$ will occur.
As we show in Section \ref{s:v_+=v_-}, the occurrence of such a  delay  contradicts the definition of $v_+$.

The proof of the above lemma is based once again on a renormalization scheme.
However, this time, we use a much simpler scale progression than the one given by \eqref{e:L_k}.
Fixed $H \geq 1$, we define
\begin{equation}
 \label{e:q_k}
  q_k = q_k^{(H)} := \sup_{w \in [0, 1) \times \{0\}} \mathbb{P}
  [w \text{ is not $(H,3^k)$-threatened}
  ].
\end{equation}

The proof of Lemma~\ref{l:threatened} will follow once we establish a fast decay rate for the sequence $q_k$ as we increase $k$ (more precisely, we will show that $q_k \leq (1/2) 3^{-\alpha k}$).
However, we first need to prove that it decays at a certain uniform rate and only then we will be able to bootstrap this to a fast decay rate resulting in Lemma \ref{l:threatened}.

\nc{c:little_threatened}
\begin{lemma}
  \label{l:little_threatened}
 Suppose Assumptions~\ref{hyp:invariant}, \ref{hyp:v+v-} and \ref{hyp:FT} are satisfied, as well as the decoupling property $\mathscr{D}(\uc{c:decouple}, \alpha)$ for some $\alpha \geq 1$  and let $\delta$ be given as in \eqref{e:delta}.
There exists an integer constant $\uc{c:little_threatened} = \uc{c:little_threatened}(\delta)$ such that for every $H \geq \uc{c:H_lower}$ if  we denote $q_{k}^{(H)} = q_k$ as in \eqref{e:q_k}, then
  \begin{equation}
    \label{e:little_threatened}
    q_{\uc{c:little_threatened} + k} \leq \big( (1 - \uc{c:some_trapped})^{1/2} \vee (1/3) \big)^k,
  \end{equation}
  for any $k \geq 2$.
\end{lemma}

In the above lemma, the rate of decay is not important, because it will be boosted soon in the proof of Lemma~\ref{l:threatened}.
Note, however, that the constant $\uc{c:little_threatened}$ only depends on the parameters of the model (including $\delta$) thus the above bound is uniform on $H \geq \uc{c:H_lower}$.
This uniformity will be useful;  if it was not needed, we could have simply used the ergodic theorem to obtain that $q_k$ vanishes with $k$.

\begin{proof}
  To simplify notations, let $1-\tilde{\uc{c:some_trapped}} := (1 - \uc{c:some_trapped})^{1/2} \vee (1/3) $.
  Thanks to Lemma~\ref{l:some_trapped}, we have $\tilde{\uc{c:some_trapped}}>0$.
  Thus we can choose an integer $\uc{c:little_threatened} > 0$  (which does not depend on $H$) for which
  \begin{equation}
    \label{e:choose_little_threatened}
    \uc{c:decouple} (1-\tilde{\uc{c:some_trapped}})^{\uc{c:little_threatened} - 1} \leq \tilde{\uc{c:some_trapped}}.
  \end{equation}

As a simple consequence of \eqref{e:some_trapped} we have
  \begin{equation}
    q_{\uc{c:little_threatened} + 2} \leq (1 - \uc{c:some_trapped}) \leq \big( (1 - \uc{c:some_trapped})^{1/2} \vee (1/3) \big)^2= (1-\tilde{\uc{c:some_trapped}})^2,
  \end{equation}
  proving \eqref{e:little_threatened} for the case $k = 2$.

  Suppose now that we have established \eqref{e:little_threatened} for some value of $k \geq 2$ and let us show that it also holds for $k + 1$.

  Observe first that, if for some $w \in [0, 1) \times \{0\}$, the event
  \begin{equation}
    \bigcap_{j=0}^{3^{\uc{c:little_threatened}+k + 1} - 1}[\text{$w + j H (v_+, 1)$ is not $H$-trapped}]
  \end{equation}
  occurs, then both
  \begin{equation}
  \label{e:small_stick}
    \bigcap_{j=0}^{3^{\uc{c:little_threatened}+k} - 1}
    \Big[
    \begin{array}{c}
      \text{$w + j H (v_+, 1)$ is not}\\
      \text{ $H$-trapped}
    \end{array}
    \Big] \,\,  \text{ and } \,\,
    \bigcap_{j=2 \cdot 3^{\uc{c:little_threatened}+k}}^{3^{\uc{c:little_threatened}+k+1} - 1}
    \Big[
    \begin{array}{c}
      \text{$w + j H (v_+, 1)$ is not}\\
      \text{$H$-trapped}
    \end{array}
    \Big]
  \end{equation}
  occur.
  Notice that the two events in \eqref{e:small_stick} are measurable with respect to the environment together with the arrival times $(T_i^x)$ and $(U_i^x)$ corresponding to space-time points contained in suitable boxes of side-length at most $5\cdot 3^{\uc{c:little_threatened}+k}H$ separated vertically by a distance of $3^{\uc{c:little_threatened}+k}H$ (for instance one can take the boxes $B^1_{3^{\uc{c:little_threatened}+k}H}(w)$ and $2\cdot 3^{\uc{c:little_threatened}+k}H (v_+,1)+B^1_{3^{\uc{c:little_threatened}+k}H}(w)$).

Thus we can use $\mathscr{D}(\uc{c:decouple}, \alpha)$ to deduce that
  \begin{equation}
    \label{e:q_induction_inequality}
    q_{\uc{c:little_threatened}+k + 1} \leq q_{\uc{c:little_threatened}+k}^2 + \uc{c:decouple} \big( 3^{\uc{c:little_threatened}+k} H \big)^{-\alpha} \leq q_{\uc{c:little_threatened}+k}^2 + \uc{c:decouple} 3^{-\alpha(\uc{c:little_threatened}+ k)}.
  \end{equation}
  Joining this with the fact that we know the validity of \eqref{e:little_threatened} for $k$, we get
  \begin{equation}
    \begin{split}
      \frac{q_{\uc{c:little_threatened} + k + 1}}{ (1-\tilde{\uc{c:some_trapped}})^{k + 1}}
      & \overset{(\alpha \geq 1)}\leq (1-\tilde{\uc{c:some_trapped}})^{-k - 1}
      \big( q_{\uc{c:little_threatened} + k}^2
      + \uc{c:decouple} 3^{-(\uc{c:little_threatened} + k)} \big)\\
      & \leq(1-\tilde{\uc{c:some_trapped}})^{k - 1} + \uc{c:decouple}  (1-\tilde{\uc{c:some_trapped}})^{\uc{c:little_threatened} - 1} \overset{k \geq 2, \eqref{e:choose_little_threatened}}\leq 1.
    \end{split}
  \end{equation}
  This finishes the proof of the lemma by induction.
\end{proof}

We can now prove Lemma~\ref{l:threatened}.

\begin{proof}[Proof of Lemma~\ref{l:threatened}]
  We first choose an integer $\dot{c} \geq 1$ such that
  \begin{equation}
    \label{e:choose_threatened}
    2 \uc{c:decouple} 3^{-\alpha(\dot{c} - 1)} \leq \frac{1}{2}.
  \end{equation}
  Observe also that from Lemma~\ref{l:little_threatened} there exists an integer $c'(\delta) \geq \uc{c:little_threatened} (\delta)\vee \dot{c}$ such that
  \begin{equation}
    q_{c' + 1} = q^{(H)}_{c' + 1} \leq \frac{1}{2} 3^{-\alpha}, \text{ uniformly on $H \geq \uc{c:H_lower}$}.
  \end{equation}

  Our aim is to show by induction that
  \begin{equation}
    \label{e:induction_threatened}
    q_{c' + k} \leq \frac{1}{2} 3^{-\alpha k}, \text{ for every $k \geq 1$},
  \end{equation}
  again, uniformly on $H \geq \uc{c:H_lower}$, which has already been established for $k = 1$.

  Suppose that \eqref{e:induction_threatened} has already been established for some $k \geq 1$.
  Then, using an argument similar to that leading to \eqref{e:q_induction_inequality}, we obtain that
  \begin{equation}
    \begin{split}
      \frac{q_{c' + k + 1}}{\tfrac{1}{2} 3^{-\alpha (k + 1)}}
      & \leq 2 \cdot 3^{\alpha (k + 1)} \Big( \frac{1}{4} 3^{-2 \alpha k}
      + \uc{c:decouple} 3^{-\alpha (c' + k)} \Big)\\
      & \leq \underbrace{\frac{1}{2} 3^{-\alpha(k - 1)}}_{\leq 1/2}
      + \underbrace{2 \uc{c:decouple} 3^{-\alpha (c' - 1)}}_{\leq 1/2 \text{ by \eqref{e:choose_threatened}}} \leq 1.
    \end{split}
  \end{equation}
  This proves \eqref{e:induction_threatened} by induction.

  Now let $r > 3^{c'}$ and fix $k \geq 0$ such that $3^{c'+k} \leq r < 3^{c'+k+1}$.
  Thus,
  \begin{equation}
  \begin{split}
    \mathbb{P} &[\text{$0$ is not $(H, r)$-threatened}] \\
     & \leq \sup_{w \in [0, 1) \times \{0\}} \mathbb{P}
  [
 w\text{ is not $(H,3^{c'+k})$-threatened}
  ] \\
  & \leq \frac{1}{2}3^{-\alpha k} \leq \frac{3^{\alpha(c'+1)}}{2} r^{-\alpha}.
  \end{split}
  \end{equation}
  By properly choosing the constant $\uc{c:threatened}(\delta)$ in order to accommodate small values of $r$, the proof is finished.
\end{proof}

\subsection{Threatened paths}
\label{s:threat_paths}

We already know that a threatened point will most likely cause a delay to the random walker and that the probability that a point is $(H,r)$-threatened can be made arbitrarily high by increasing $r$, uniformly in $H$.
This section is dedicated to the task of showing that the trajectory of the random walker cannot avoid threatened points.
Since we still know very little about the actual behavior of the random walker trajectory, we instead show that, with high probability, every allowed path spends a significant  proportion of its time on threatened points.

Recall the Definition~\ref{d:allowed} of allowed paths and the scale sequence $L_k$ in \eqref{e:L_k}.
From now on we are always going to consider $(H,r)$-threatened points with pairs $(H,r)$ chosen so that  $H = h L_k$ and $r = l_k$ for some positive integer $k$.

We start by proving that, for large enough $k$, with high probability, every point in $I^h_{L_k}$ lies close to a $(hL_k,l_k)$-threatened point.
More precisely,

\nc{c:L_k2-threat}
\begin{lemma}
  \label{l:trigger_s}
  Recall the definition of $\uc{c:F_m}$ from \eqref{e:F_m_decay}.
  If $\alpha\geq 8$, there exists an integer $\uck{k:threat_path} = \uck{k:threat_path} (\delta)$ and a constant $\uc{c:L_k2-threat}=\uc{c:L_k2-threat}(\delta)>0$ such that
  \begin{equation}
   \label{e:k_big_enough}
  L_{\uck{k:threat_path}} > \uc{c:H_lower},
  \end{equation}
    \begin{equation}
    \label{e:trigger_s}
    \mathbb{P} \Big[
    \begin{array}{c}
      \text{there exists some $y \in I^h_{L_{\uck{k:threat_path} + 1}}(w)$ such that}\\
      \text{$\lfloor y \rfloor_{h L_{\uck{k:threat_path}}}$ is not $(h L_{\uck{k:threat_path}}, l_{\uck{k:threat_path}})$-threatened}
    \end{array}
    \Big] \leq \uc{c:L_k2-threat} L_{\uck{k:threat_path} + 1}^{-(\alpha-1)/5}
  \end{equation}
  uniformly over $h \geq 1$ and $w \in \mathbb{L}$ and
  \begin{equation}
    \label{e:k_trigger_large}
    25 (\uc{c:L_k2-threat}^2+\uc{c:decouple}) L_{k}^{(23-3\alpha)/20} + 25\, \uc{c:F_m}^{-1} L_{k}^{(\alpha+3)/4} e^{- \uc{c:F_m} L_{k}} \leq 1 \text{ for every $k \geq \uck{k:threat_path}$}.
  \end{equation}
  \end{lemma}

\begin{proof}
It is obvious that \eqref{e:k_big_enough} holds for $\uck{k:threat_path}$ large enough.
  A direct application of Lemma~\ref{l:threatened} yields
  \begin{equation}
   \label{e:hlklk_threatened}
    \mathbb{P} \big[ 0 \text{ is not $(hL_k, l_k)$-threatened} \big] \leq \uc{c:threatened} l_k^{-\alpha},
  \end{equation}
  for every $k$ such that $L_k \geq \uc{c:H_lower}>4/\delta$,  uniformly over $h \geq 1$ (recall that $\uc{c:H_lower}$ only depends on $\delta$ and on the environment).

Let $w \in \mathbb{L}$.
Knowing that $\pi_1(\lfloor y \rfloor_{h L_k})$ is an integer for each $y \in I^h_{L_{k+1}}(w)$ and using \eqref{e:hlklk_threatened} together with translation invariance we get that, for some suitable constant $\uc{c:L_k2-threat}>0$ depending only on $\delta$ (and in the environment),
  \begin{equation}
    \label{e:all_are_threatened}
    \begin{split}
    \mathbb{P} & \Big[
    \begin{array}{c}
      \text{there exists some $y \in I^h_{L_{k + 1}}(w)$ such that}\\
      \text{$\lfloor y \rfloor_{h L_k}$ is not $(h L_{k}, l_{k})$-threatened}
    \end{array}
    \Big]\\
      & \leq\left\lceil \frac{h L_{k + 1}}{\lfloor(\delta/4)h L_k\rfloor}\right\rceil \; (\uc{c:threatened} l_{k}^{-\alpha}) \leq c(\delta)\, l_k^{-\alpha + 1} \leq \uc{c:L_k2-threat}\, L_{k+1}^{-(\alpha-1)/5}
    \end{split}
  \end{equation}
for every $k$ such that $L_k \geq \uc{c:H_lower}>4/\delta$,  uniformly over $h \geq 1$.

Now that $\uc{c:L_k2-threat}$ is fixed, let us consider \eqref{e:k_trigger_large}.
Since we are assuming $\alpha \geq 8$, the exponent $(23-3\alpha)/20$ appearing in the left-hand side is negative.
Therefore, \eqref{e:k_trigger_large} holds as soon as $k$ is sufficiently large.
This concludes the proof.
\end{proof}

\begin{remark}
  \label{r:rounding}
  Notice that we have only considered rounded points $\lfloor y \rfloor_{h L_k}$.
  This was crucial for the conclusion of Lemma~\ref{l:trigger_s} to hold uniformly over all integers $h$.
  Indeed, if we had considered every integer $y\in I^h_{L_{k+1}}(w)$, the factor $h$ would not have cancelled out in \eqref{e:all_are_threatened}.
  This shows that the reason for introducing the rounding procedure in equation \eqref{e:round_H} is to lower the entropy when looking for non-threatened points inside boxes.
\end{remark}

From now on we will keep $\uck{k:threat_path}$ fixed as in Lemma~\ref{l:trigger_s} and we will check whether certain points are $(hL_{\uck{k:threat_path}}, l_{\uck{k:threat_path}})$-threatened.
We cannot hope that the random walk will always be close to a threatened point.
We instead look at the density of time that the random walk spends around threatened points as made precise in the following definition.

\begin{definition}
  Fix $h \geq 1$ and let $\uck{k:threat_path}$ be as in Lemma~\ref{l:trigger_s}.
Given some $k \geq \uck{k:threat_path} + 1$ and an allowed path $Y=(Y_t)_{t \in [0, h L_k]}$, we define its threatened density as
  \begin{equation*}
    D^h(Y) := \frac{1}{L_k / L_{\uck{k:threat_path} + 1}} \# \Big\{ 0 \leq j < \tfrac{L_{k}}{L_{\uck{k:threat_path} + 1}} \colon \lfloor Y_{j h L_{\uck{k:threat_path} + 1}} \rfloor_{h L_{\uck{k:threat_path}}} \text{ is $(h L_{\uck{k:threat_path}}, l_{\uck{k:threat_path}})$-threatened} \Big\}.
  \end{equation*}
\end{definition}

Note that on the complementary of the event appearing in the Eq.\ \eqref{e:trigger_s} in Lemma~\ref{l:trigger_s}, every allowed path $Y=(Y_t)_{t \in [0, h L_{k_2+1}]}$ starting at $I^h_{L_{\uck{k:threat_path}+1}}(w)\cap \mathbb{L}$ has $D^h(Y)$ equal to one.
More precisely,
\begin{equation}
  \begin{split}
    & \Big[
    \begin{array}{c}
      \text{there exists some $y \in I^h_{L_{\uck{k:threat_path} + 1}}(w)$ such that}\\
      \text{$\lfloor y \rfloor_{h L_{\uck{k:threat_path}}}$ is not $(h L_{\uck{k:threat_path}}, l_{\uck{k:threat_path}})$-threatened}
    \end{array}
    \Big]^{\mathsf{c}}\\
    & \qquad \subseteq \Big[
    \begin{array}{c}
      \text{every allowed path $(Y_t)_{t \in [0, h L_{\uck{k:threat_path} + 1}]}$ starting at}\\
      \text{$I^h_{L_{\uck{k:threat_path} + 1}}(w) \cap \mathbb{L}$ satisfies $D^h(Y) = 1$}
    \end{array}
    \Big].
  \end{split}
\end{equation}
Indeed, according to the definition of $D^h(Y)$, for $k = \uck{k:threat_path} + 1$ one only needs to check whether the starting point of $Y$ is $(h L_{\uck{k:threat_path}}, l_{\uck{k:threat_path}})$-threatened.

Considering the above remark, Lemma~\ref{l:trigger_s} establishes that at scale $\uck{k:threat_path} + 1$, with high probability all the paths starting at $I^h_{L_{\uck{k:threat_path} + 1}} \cap \mathbb{L}$  have $D^h(Y) = 1$.
But, as we have already observed, the random walker path will eventually pass through regions composed of non-threatened points which could cause the threatened density to drop under one.

The next lemma, which is the main result of this section, shows that, with high probability, every allowed path spends a positive proportion of its time next to threatened points.

\nck{k:threat_path}
\begin{lemma}[Threatened paths]
  \label{l:threatened_paths}
  Assume $\alpha \geq 8$.
  Then, for any integer $k \geq \uck{k:threat_path} + 1$, we have
  \begin{equation}
  \label{e:threat_path_bound}
    \mathbb{P} \Big[
    \begin{array}{c}
      \text{there exists an allowed path $Y = (Y_t)_{t \in [0, h L_k]}$}\\
      \text{starting at $I^h_{L_k}(w) \cap \mathbb{L}$ and having $D^h(Y) < 1/2$}
    \end{array}
    \Big] \leq \uc{c:L_k2-threat} L_k^{-(\alpha-1)/5}.
  \end{equation}
  uniformly in $h \geq 1$ and $w \in \mathbb{R}^2$.
\end{lemma}

\nck{k:tail}
Similarly to what we have done in the definition of the speeds $v_k$ in Section~\ref{s:deviations}, we are going to introduce a sequence of densities that will always remain above $1/2$.
Note that by our choice of scales (see \eqref{e:L_k})
\begin{equation}
  \sum_{k \geq 1} \frac{2}{l_k} \leq \frac{1}{2},
\end{equation}
so that if we define
\begin{equation}
  \rho_{\uck{k:threat_path}} := 1 \quad \text{and} \quad \rho_{k + 1} := \rho_k - \frac{2}{l_k} \text{ for $k \geq \uck{k:threat_path}$},
\end{equation}
we have $\rho_k \geq 1/2$ for every $k \geq \uck{k:threat_path}$.

\begin{proof}[Proof of Lemma~\ref{l:threatened_paths}]
  We use a renormalization scheme based on an induction on $k \geq \uck{k:threat_path} + 1$.
  The case $k = \uck{k:threat_path} + 1$ has already been dealt with in Lemma~\ref{l:trigger_s}.

We introduce hierarchical events, similar to those appearing in \eqref{e:threat_path_bound}.
  For this, given $k \geq \uck{k:threat_path} + 1$ and an index $m \in M^h_k$ (recall the definition in \eqref{e:M_h_k}) we define
  \begin{equation}
    S_m := \Big[
    \begin{array}{c}
      \text{there exists an allowed path $Y = (Y_t)_{t \in [0, h L_k]}$}\\
      \text{starting at $I_m \cap \mathbb{L}$ and satisfying $D^h(Y) \leq \rho_k$}
    \end{array}
    \Big]
  \end{equation}
  and write
  \begin{equation}
    \label{e:s_k}
    s^h_k := \sup_{m \in M^h_k} \mathbb{P} \big[ S_m \big].
  \end{equation}

  Since all the densities $\rho_k$ are at least equal to $1/2$, it is enough to show that
  \begin{equation}
    \label{e:s_k_decay}
    s^h_k \leq \uc{c:L_k2-threat} L_k^{-(\alpha-1)/5}, \text{ for every $k \geq \uck{k:threat_path} + 1$,}
  \end{equation}
 uniformly over $h \geq 1$.

  Observe that by Lemma~\ref{l:trigger_s} we already have
  \begin{equation}
    s^h_{\uck{k:threat_path} + 1} \leq \uc{c:L_k2-threat} L^{-(\alpha-1)/5}_{\uck{k:threat_path} + 1},
  \end{equation}
  uniformly over $h \geq 1$.
  Therefore, from now on we assume that \ref{e:s_k_decay} holds for some $k \geq \uck{k:threat_path} + 1$ and prove that it also holds for $k + 1$.

  Recall the definition of the events $F_m$ in \eqref{e:boundedspeed}.
 Using the exact same argument as in the proof of Lemma~\ref{l:cascade}, we can show that for $m \in M^h_{k + 1}$ with $k \geq \uck{k:threat_path}$
   \begin{display}
    on the event $S_m \cap \big( \cap_{m' \in C_m} F_{m'} \big)$ there exist $m_1, m_2$ in $C_m$ such that $S_{m_1}\cap F_{m_1}$ and $S_{m_2}\cap F_{m_2}$ occur and $d(B_{m_1}, B_{m_2}) \geq  h L_k$.
  \end{display}
In fact, if $m=(h,k+1,w) \in M_{k+1}$, one can split the time interval $\pi_2(w)+[0,hL_{k+1}]$ into $l_k$ layers of length $L_k$.
An allowed path $Y$ starting at $I_m$, crosses each of the $l_k$ layers starting from a point in an interval of the type $I_{m_i}$ with $m_i \in C_m$.
Assume that, for at most two of these layers, the event $S_{m_i}$ occurs at the corresponding index $m_i$.
Then we would have $D^h(Y)\geq \rho_k-2\rho_k/l_k>\rho_{k+1}$ so that $S_m$ could not occur.
Therefore $S_{m_i}$ has to occur for at least three layers which allows us to find the boxes $B_{m_1}$ and $B_{m_2}$ with time separation at least equal to $hL_k$.

  Note also that $S_{m_i} \cap F_{m_i}$ is measurable with respect to the environment inside $B_{m_i}$ together with the arrival times $(T_i^x)$ and the random variables $(U^x_i)$ associated to space-time points inside $B_{m_i}$.

  Therefore using the fact that ${B}_{m_1}$ and ${B}_{m_2}$ are boxes of side lengths at most $5h L_k$ separated by a time-distance at least equal to $hL_k$, that $\mathbb{P}$ is stationary, invariant under shifts by $\mathbb{L}$ and using $\mathscr{D}(\uc{c:decouple}, \alpha)$ we conclude that
  \begin{equation}
    \mathbb{P} \big( (S_{m_1} \cap F_{m_1}) \cap (S_{m_2} \cap F_{m_2}) \big) \leq (s_k^h)^2 + \uc{c:decouple} L_k^{-\alpha}.
  \end{equation}

  With this, we can estimate
  \begin{equation}
    \begin{split}
      \frac{s^h_{k + 1}}{L_{k + 1}^{-(\alpha-1)/5}} &
      \leq  L_k^{(\alpha-1)/4} \big( 25\, l_k^4 \,((s^h_k)^2 + \uc{c:decouple} L_k^{-\alpha}) + 5l_k^2\, \uc{c:F_m}^{-1} e^{- \uc{c:F_m} L_k}   \big)\\
      & \leq  25\, l_k^4  L_k^{(\alpha-1)/4} \big( (s^h_k)^2 + \uc{c:decouple} L_k^{-\alpha} + \uc{c:F_m}^{-1} e^{- \uc{c:F_m} L_k}   \big)\\
      & \leq 25 L_k^{1+(\alpha-1)/4 } \big( \uc{c:L_k2-threat}^2 L_k^{-2 (\alpha-1)/5} + \uc{c:decouple} L_k^{-\alpha} +  \uc{c:F_m}^{-1} e^{-\uc{c:F_m} L_k} \big)\\
      & \leq 25 (\uc{c:L_k2-threat}^2+ \uc{c:decouple}) L_k^{(23-3\alpha)/20} + 25 \uc{c:F_m}^{-1} L_k^{(\alpha+31)/4} e^{- \uc{c:F_m} L_k}\overset{ \eqref{e:k_trigger_large}}\leq 1,
    \end{split}
  \end{equation}
  concluding the proof of Lemma~\ref{l:threatened_paths}.
\end{proof}

\section{Proof of Theorem~\ref{t:main}}\label{s:v_+=v_-}

We are now ready to prove the main result of this article.

\begin{proof}[Proof of Theorem~\ref{t:main}]
  In view of Lemma~\ref{l:deviation_interp}, it is enough to prove that $v_- = v_+$.
  Suppose by contradiction that $v_- < v_+$ and let
  \begin{equation}
    \delta = \frac{v_+ - v_-}{4} \quad \text{and} \quad \eta := \frac{\delta}{4l_{ \uck{k:threat_path}}},
  \end{equation}
 where $\uck{k:threat_path}$ is given as in Lemma~\ref{l:trigger_s}.

  From now on, given an index $k\geq \uck{k:threat_path}+1$ we will choose $h = h_k = L_k$ so that the quantities $h L_k$ and $h L_{\uck{k:threat_path}}$ that appeared often in the previous section will be turned into $L_k^2$ and $L_k L_{\uck{k:threat_path}}$, respectively.
  Our aim is to show that
  \begin{equation}
    \label{e:then_delay}
    \lim_{k \to \infty} p_{L_k^2} (v_+ - \eta/2)= 0,
  \end{equation}
  which contradicts the definition of $v_+$.

Let us start by proving that allowed paths usually do not exceed average speed $v_+ + \eta$.
  More precisely, given some $k \geq \uck{k:threat_path} + 1$ and $w \in \mathbb{R}^2$, we consider the box $B^{L_k}_{L_k}(w) = (w+[-2 L_k^2, 3 L_k^2] \times [0, L_k^2]) $ and slice it along the sequence of time steps
  \begin{equation}
    J_{0,w} = \pi_2(w) + \{0, L_k L_{\uck{k:threat_path}}, 2 L_k L_{\uck{k:threat_path}}, \dots, (L_k/L_{\uck{k:threat_path}}-1)L_k L_{\uck{k:threat_path}}\},
  \end{equation}
  which contains $L_k/L_{\uck{k:threat_path}}$ elements.
  We want a lower bound on the probability of the following event:
  \begin{equation}
    \label{e:event_no_hurry}
    G_1(w) := \Big[
    \begin{array}{c}
      \text{for every $y \in B^{L_k}_{L_k}(w)\cap\mathbb{L}$ with $\pi_2(y) \in J_{0,w}$}\\
      \text{ $X^y_{L_k L_{\uck{k:threat_path}}} - \pi_1(y) \leq \big( v_+ + \eta \big) L_k L_{\uck{k:threat_path}}$}
    \end{array}
    \Big].
  \end{equation}
By paving the box $B^{L_k}_{L_k}(w)$ with boxes of side length $5 L_k L_{\uck{k:threat_path}}$ by $L_k L_{\uck{k:threat_path}}$ and using Lemma~\ref{l:deviation_interp}, we conclude that
  \begin{equation}
    \label{e:no_hurry}
   \sup_{w\in[0,1)\times\{0\}} \mathbb{P} (G_1(w)^{\mathsf{c}}) \leq c \; \Big( \frac{L_k}{L_{\uck{k:threat_path}}} \Big)^2 \big( L_k L_{\uck{k:threat_path}} \big)^{-\alpha/4}
  \end{equation}
  which converges to zero as $k$ goes to infinity, since we are assuming $\alpha > 8$.

Roughly speaking, inequality \eqref{e:no_hurry} shows that that the random walker cannot hurry up too much in any of the time subintervals of length $L_k L_{\uck{k:threat_path}}$.
  Moreover, we know from Lemma~\ref{l:threatened_paths} that it typically spends a large proportion of its time on threatened points.
Indeed, let us denote
\begin{equation}
G_2(w) := \Big[
      \begin{array}{c}
        \text{every allowed path $Y=(Y_t)_{t \in [0, L_k^2]}$ starting at}\\
        \text{$w + [0, L_k^2) \times \{0\} \cap \mathbb{L}$ satisfies $D^{L_k}(Y) \geq 1/2$}
      \end{array}
      \Big]
\end{equation} and use Lemma~\ref{l:threatened_paths} in order to get
  \begin{equation}
   \sup_{w\in [0,1)\times \{0\}}   \mathbb{P} (G_2(w)) \geq 1 - \uc{c:L_k2-threat} L_k^{-(\alpha-1)/5}.
  \end{equation}

  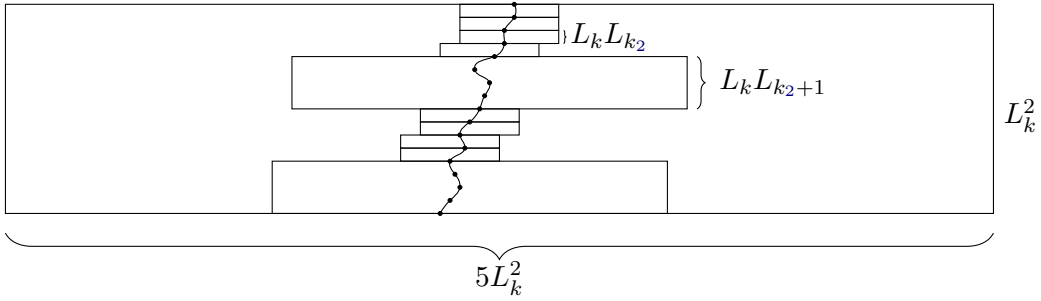
\begin{figure}
    \centering
      \begin{tikzpicture}[use Hobby shortcut,scale=1.3]
        \draw (0, 0) rectangle (10, 2.1333);
        \draw (4.40, 0) .. (4.50, 0.1333) .. (4.60, 0.2666) .. (4.55, 0.4) .. (4.50, 0.5333) ..          (4.65, 0.6666) .. (4.60, 0.8) .. (4.70, 0.9333) .. (4.80, 1.0666) .. (4.85, 1.2) .. (4.90, 1.3333) .. (4.75, 1.4666) .. (4.95, 1.6) .. (5.05, 1.7333) .. (5.05, 1.8666) .. (5.15, 2) .. (5.15, 2.1333);
        \draw[fill] (4.40, 0) circle (0.02);
        \draw[fill] (4.50, 0.1333) circle (0.02);
        \draw[fill] (4.60, 0.2666) circle (0.02);
        \draw[fill] (4.55, 0.4) circle (0.02);
        \draw[fill] (4.50, 0.5333) circle (0.02);
        \draw[fill] (4.65, 0.6666) circle (0.02);
        \draw[fill] (4.60, 0.8) circle (0.02);
        \draw[fill] (4.70, 0.9333) circle (0.02);
        \draw[fill] (4.80, 1.0666) circle (0.02);
        \draw[fill] (4.85, 1.2) circle (0.02);
        \draw[fill] (4.90, 1.3333) circle (0.02);
        \draw[fill] (4.75, 1.4666) circle (0.02);
        \draw[fill] (4.95, 1.6) circle (0.02);
        \draw[fill] (5.05, 1.7333) circle (0.02);
        \draw[fill] (5.05, 1.8666) circle (0.02);
        \draw[fill] (5.15, 2) circle (0.02);
        \draw[fill] (5.15, 2.1333) circle (0.02);
        \draw (2.70, 0) rectangle (6.70, 0.5333); 
        \draw (4.00, 0.5333) rectangle (5.00, 0.6666);
        \draw (4.00, 0.6666) rectangle (5.00, 0.8);
        \draw (4.20, 0.8) rectangle (5.20, 0.9333);
        \draw (4.20, 0.9333) rectangle (5.20, 1.0666);
        \draw (2.90, 1.0666) rectangle (6.90, 1.6); 
        \draw (4.40, 1.6) rectangle (5.40, 1.7333);
        \draw (4.60, 1.7333) rectangle (5.60, 1.8666);
        \draw (4.60, 1.8666) rectangle (5.60, 2);
        \draw (4.60, 2) rectangle (5.60, 2.1333);
        \draw[right] (10, 1) node {$L_k^2$};
        \draw [decorate,decoration={brace,amplitude=10pt}] (10, -.2) -- (0, -.2) node [black, midway, yshift=-0.6cm] {$5 L_k^2$};
        \draw [decorate,decoration={brace,amplitude=3pt}] (7, 1.6) -- (7, 1.0666) node [black, midway, xshift=1cm] {$L_k L_{\uck{k:threat_path} + 1}$};
        \draw [decorate,decoration={brace,amplitude=.8pt}] (5.65, 1.8666) -- (5.65, 1.7333) node [black, midway, xshift=.6cm] {$L_k L_{\uck{k:threat_path}}$};
      \end{tikzpicture}
    \caption{The final bound in the proof of Theorem~\ref{t:main}. The large boxes correspond to the displacements for $j \in J'$, while the...}
    \label{f:J_y}
  \end{figure}

  Given $w\in \mathbb{R}^2$ and $y \in (w+ [0, L_k^2) \times \{0\} )\cap \mathbb{L}$, note that $(Y^y_t)_{t \in [0, L_k^2]}$ is an allowed path.
  We denote by $J^y \subseteq \{0, \dots, L_k / L_{\uck{k:threat_path}+1} - 1\}$ the set of indices $j$ for which the point $\lfloor Y^y_{j L_k L_{\uck{k:threat_path}+1}}\rfloor_{L_k L_{\uck{k:threat_path}}}$ is $(L_k L_{\uck{k:threat_path}}, l_{\uck{k:threat_path}})$-threatened so that, on the event $G_2(w)$, the set $J^y$ has at least $L_k/(2 L_{\uck{k:threat_path}+1})$ elements.

  Suppose now that $G_1 (w) \cap G_2(w)$ occurs.
  Then, given a point $y \in (w + [0, L_k^2) \times \{0\}) \cap \mathbb{L}$, we can estimate
  \begin{align*}
    X^y_{L_k^2} - \pi_1(y) & = \sum_{j = 0}^{L_k / L_{\uck{k:threat_path}+1} - 1} X^{y}_{(j + 1) L_k L_{\uck{k:threat_path} + 1}} - X^y_{j L_k L_{\uck{k:threat_path} + 1}}\\
    & = \sum_{j \in J^y} \big[ X^{y}_{(j + 1) L_k L_{\uck{k:threat_path} + 1}} - X^y_{j L_k L_{\uck{k:threat_path} + 1}} \big] \\
&\quad+\ \sum_{j \not \in J^y} \big[X^{y}_{(j + 1) L_k L_{\uck{k:threat_path} + 1}} - X^y_{j L_k L_{\uck{k:threat_path} + 1}}\big].
\end{align*}
 Since we are on $G_1(w)$, $X^{y}_{(j + 1) L_k L_{\uck{k:threat_path} + 1}} - X^y_{j L_k L_{\uck{k:threat_path} + 1}} \leq (v_+ + \eta) L_{\uck{k:threat_path}+1} $, so
\begin{align*}
 X^y_{L_k^2} - \pi_1(y)   & \leq \sum_{j \in J^y} \big[X^{y}_{(j + 1) L_k L_{\uck{k:threat_path} + 1}} - X^y_{j L_k L_{\uck{k:threat_path} + 1}}\big]
  \\&\quad+\ \Big( \frac{L_k}{L_{\uck{k:threat_path} + 1}} - |J^y| \Big) (v_+ + \eta) L_k L_{\uck{k:threat_path} + 1}.
\end{align*}
 Now, for $j \in J^y$, $\lfloor X^y_{jL_k L_{\uck{k:threat_path}}}\rfloor$ is $(L_k L_{\uck{k:threat_path}},l_{\uck{k:threat_path}})$-threatened.
    Furthermore, since $G_1(w)$ occurs, Lemma \ref{l:delays} guarantees that \[X^{y}_{(j + 1) L_k L_{\uck{k:threat_path} + 1}} - X^y_{j L_k L_{\uck{k:threat_path} + 1}} \leq (v_+ - \delta/2l_{\uck{k:threat_path}})L_k L_{\uck{k:threat_path}}\] and we can estimate
\begin{align*}
  X^y_{L_k^2} - \pi_1(y)  & \leq |J^y| \big( v_+ - \delta/(2l_{\uck{k:threat_path}}) \big) L_k L_{\uck{k:threat_path} + 1} + \Big( \frac{L_k}{L_{\uck{k:threat_path} + 1}} - |J^y| \Big) (v_+ + \eta) L_k L_{\uck{k:threat_path} + 1} \\
    & \leq v_+L_k^2 - |J^y|(\delta/2l_{\uck{k:threat_path}}) L_k L_{\uck{k:threat_path} + 1} + \Big( \frac{L_k}{L_{\uck{k:threat_path} + 1}} - |J^y| \Big)  \eta L_k L_{\uck{k:threat_path} + 1}\\
    & \leq  v_+ L_k^2 - (\delta/4l_{\uck{k:threat_path}} - \eta/2)L_k^2   = \big( v_+ - \eta/2 \big) L_k^2.
  \end{align*}

   The fact that $\sup_{w\in[0,1)\times\{0\}}\mathbb{P}(G_1 (w) \cap G_2(w))$ converges to one proves \eqref{e:then_delay} which, in turn, contradicts the definition of $v_+$.
  This proves that $v_- = v_+$ and, consequently, the proof of Theorem~\ref{t:main} follows immediately from Lemma~\ref{l:deviation_interp}.
\end{proof}

\section{Proof of Theorem~\ref{t:positive}}\label{s:positive}

In this section we prove Theorem~\ref{t:positive}, which gives us conditions to conclude that the speed of the random walker is strictly positive.

\nc{c:no_stop}
\begin{proof}[Proof of Theorem~\ref{t:positive}]
  Since the hypotheses of Theorem~\ref{t:main} are satisfied, we already have a law of large numbers for the random walker, as well as a deviations bound for its asymptotic speed.
  Therefore, all we have to do is to establish the sign of the random walker's speed.

  For this proof we will need two exponents $\beta$ and $\gamma$ satisfying
  \begin{equation}
    \label{e:beta_gamma}
    \beta \in \Big( 5, \alpha -\frac{14}{4} \Big) \qquad \text{and} \qquad \gamma \in \Big( \frac{4 + \beta}{2 + 4 \alpha}, \frac{1}{4} \Big).
  \end{equation}
  The need for these exact requirements will become clear later in the proof.
  For now, all one needs to observe is that this is possible since we assumed that $\alpha > 5+14/4=8.5$: indeed, one can choose $\beta$ and $\beta'$ satisfying
  \begin{equation}
    5 < \beta < \beta' < \alpha-\frac{14}{4} \text{ and then take } \gamma = \frac{4 + \beta'}{2 + 4 \alpha}.
  \end{equation}
  Note that for the above we had to assume $\alpha > 8.5$, although we believe that this number has no intrinsic meaning and could be improved upon.

  The statement of the theorem contemplates two cases: random walkers that can only jump to the right or to the left.
  Of course these two cases are symmetrical, so that we can prove only one of these claims and here we consider random walkers that can only jump to the left, since this will make it easier for us to employ lemmas from Section~\ref{s:deviations}. The running assumption is thus
\begin{equation}\label{e:negativestep}
\mathbb{P}\big[X^{o}_1\leq -1\big]>0.
\end{equation}

  We start by proving that it is very hard for the random walk to remain still, or more precisely, that there exists a constant $\uc{c:no_stop} > 0$ such that
  \begin{equation}
    \label{e:has_to_move}
    \mathbb{P} \big[ X^{o}_L = 0 \big] \leq \uc{c:no_stop} L^{-\alpha},
    \text{ for every $L \geq 1$}.
  \end{equation}
  To see why the above is true, we first define
  \begin{equation}
    q_k := \mathbb{P} \big[ X^{o}_{3^k} = 0 \big], \text{ for $k \geq 0$}.
  \end{equation}
  Then, the hypothesis \eqref{e:negativestep} guarantees that $q_0 <1$.
  Moreover, since the random walker can only jump in one direction, we have that
  \begin{equation}
    \big[ X^{o}_{3^{k+1}} = 0 \big] \subseteq \big[ X^{o}_{3^k} = 0 \big] \cap \big[ X^{(0, 2 \cdot 3^k)}_{3^k} = 0 \big].
  \end{equation}
  To conclude the proof of \eqref{e:has_to_move}, we follow exactly the same arguments as in the proof of Lemma~\ref{l:threatened}, observing that we fix $H = 1$ (so that we can ignore all the statements about uniformity on $H$) and we replace $\uc{c:some_trapped}$ by $1-\mathbb{P}[X^{o}_1 = 0]$.
  With these observations in mind, the proof of Lemma~\ref{l:threatened} applies directly to show \eqref{e:has_to_move}.

  Before proving that the random walker has a negative speed, we first prove a sub-linear bound, or more precisely, we claim that
  \begin{equation}
    \label{e:move_polynomial}
    \mathbb{P} \Big[ \frac{X^{o}_L}{L} > - L^{-\gamma} \Big] \leq c L^{2 - 2\gamma - \alpha \gamma}.
  \end{equation}
  First, let us observe that the above event is contained in
  \begin{equation}
    \bigg[
    \begin{array}{c}
      \text{for some $i \in \{0, L^{\gamma}, 2 L^{\gamma}, \dots \} \cap [0, L]$, and some}\\
      \text{integer $x \in \{-L^{1-\gamma}, \dots, 0\}$, we have }
      X^{(x, iL^\gamma)}_{L^{\gamma}} = 0
    \end{array}
    \bigg].
  \end{equation}
  In fact, if we are on not on the above event, the random walker has to make at least one jump to the left every $L^{\gamma}$ steps (as long as it remains on the right of $-L^{1-\gamma}$), and therefore $X^{o}_{L} \leq -L^{1 - \gamma}$.
  Therefore we can use a union bound to show that
  \begin{equation}
    \label{e:sub_linear_bound}
    \mathbb{P} \Big[ \frac{X^{o}_L}{L} > -L^{-\gamma} \Big] \leq 5 L^{2 - 2\gamma} \mathbb{P}\big[ X^{o}_{L^{\gamma}} = 0 \big] \overset{\eqref{e:has_to_move}}\leq c L^{2 - 2\gamma - \alpha \gamma},
  \end{equation}
  establishing \eqref{e:move_polynomial}.

  We will now use a renormalization to bootstrap the statement \eqref{e:move_polynomial} (which has a vanishing speed) into our desired negative speed result.
  The strategy to show a negative upper bound on the speed of the random walker is very similar to the one used to prove Lemma~\ref{l:deviation_reduction}.
  However, our task now will be much simpler since we can already count on Lemmas~\ref{l:cascade} and \ref{l:inductive} that have been proved in Section~\ref{s:deviations}.
  But first, let us recall some notation.

  Let $l_k$ and $L_k$ be defined as in \eqref{e:L_k}, fix $h = 1$ and recall the notations $B_m$ and $I_m$ introduced right after \eqref{e:M_h_k}.
  Recall also the definitions of $A_m(v)$ in \eqref{e:A_m_v} and $p_H(v)$ in \eqref{e:p_k}.

  As we have mentioned, we are now going to use some results of Section~\ref{s:deviations}.
  For this we recall that our random environment was assumed to satisfy $\mathscr{D}(\uc{c:decouple}, \alpha)$, which clearly implies the weaker $\mathscr{D}(\uc{c:decouple}, \beta)$.
  The reason why we will make use of this weaker decoupling condition is because it makes the hypothesis of Lemma~\ref{l:inductive} weaker as well.

  Recalling that $\beta > 5$, we can see that Lemma~\ref{l:inductive} can be applied in the current context, giving that
  \begin{display}
    if for some $\tilde{k} \geq \uc{c:inductive}$ and $v \in \mathbb{R}$ we have $p^1_{\tilde{k}}(v) \leq L_{\tilde{k}}^{-\beta/2}$, then it holds that $p^1_{k}(v_k) \leq L_k^{-\beta/2}$ for every $k \geq \tilde{k}$,
  \end{display}
  where $v_k$ is defined through: $v_{\tilde{k}} = v$ and $v_{k + 1} = v_k + 8/{l_k}$, for $k \geq \tilde{k} + 1$, similarly to~\eqref{e:v_k_o}.

Note that in Lemma~\ref{l:inductive} we needed to assume $k\geq k_1(v)$, in order for the speeds $v_k$ we considered to be defined and larger than $v_+$. Here this assumption will be replaced by the second condition in \eqref{e:tilde_p_bound} below.

  In view of the above, all we need to prove now is that there exists some scale $\tilde{k} \geq \uc{c:inductive}$ and some initial speed $v < 0$, for which
  \begin{equation}
    \label{e:tilde_p_bound}
    p^1_{\tilde{k}}(v) \leq L_{\tilde{k}}^{-\beta/2} \text{ and } v < - 2 \sum_{k \geq \tilde{k}} \frac{8}{l_k},
  \end{equation}
  the last condition being important because it implies that $\sup_k v_k < 0$, leading to a negative upper bound on the speed.

  To finish the proof, let us find the initial speed $v$ and the scale $\tilde{k}$ as required.
  We first estimate the decay of the sum in \eqref{e:tilde_p_bound} by noting that, since $l_{k + 1} \geq 2 l_k$ for every $k \geq 1$ we have
  \begin{equation}
    2 \sum_{k \geq k'} \frac{8}{l_k} \leq \frac{32}{l_{k'}}.
  \end{equation}
  Therefore, if we take $v_i = -L_i^{-\gamma}$, we get
  \begin{equation}
    \label{e:p_i_decay}
    p_i^1(v_i) \leq \mathbb{P} \Big[ \frac{X_{L_i}^{o}}{L_i} > -L_i^{-\gamma} \Big] \overset{\eqref{e:sub_linear_bound}}\leq c L_i^{2 - \gamma(2 + \alpha)} \overset{\eqref{e:beta_gamma}, k \text{ large}}\leq L_i^{-\beta/2}.
  \end{equation}
  Using the fact that $\gamma < 1/4$, we conclude that for large enough $i$
  \begin{equation}
    v_i \leq - \frac{32}{l_i} \leq -2 \sum_{j \geq i} \frac{8}{l_j}.
  \end{equation}
  Finally, we use the above bound, together with \eqref{e:p_i_decay} to conclude that for some $i$ large enough we can set $v = v_i$ and $\tilde{k} = i$ satisfying all the requirements of \eqref{e:tilde_p_bound}.
  This finishes the proof of the theorem.
\end{proof}

\section{Applications}\label{s:applications}
In this section we present applications of our main results, Theorems~\ref{t:main} and \ref{t:positive}. In some cases we just prove that a given process satisfies $\mathscr{D}_{\mathrm{env}}(c_0,\alpha)$ for some $\alpha$ large enough. Remark~\ref{rem:Poissontimes} then allows to apply our results for a large class of random walkers in that environment.

\subsection{The contact process}
\label{ss:contact}

Random walks on supercritical contact process have been studied in various papers such as \cite{2012arXiv1209.1511D, mountford2015, Bethuelsen2018}.
In \cite{mountford2015}, the authors prove a Law of Large Numbers and a Central Limit Theorem for such random walks under quite general assumptions.
As a good illustration of the applicability of our methods we give a new proof of the law of large numbers for this model in dimension 1.
As we have mentioned before the proof of a central limit theorem is still beyond the scope of our techniques.

Here we refrain from introducing the full notation and some classical auxiliary results for the contact process and refer to \cite{liggett} or Section~2 of \cite{mountford2015} for its definition via graphical construction and for the proofs of some of these results.
For $x,y\in\mathbb{Z}$ and $s,t\in\mathbb{R}$, we write $(x,t)\leftrightarrow(y,s)$ if the two space-time points are connected through the percolation structure on $\mathbb{Z}\times\mathbb{R}$ induced by the graphical representation (note that time is allowed to assume negative values).
We use a similar notation for denoting connection between subsets of $\mathbb{Z}\times \mathbb{R}$.

Fix an infection rate $\lambda > 0$.
For $A\subset\mathbb{Z}$ and $t\in\mathbb{R}$ we define
\begin{equation}
\eta_s^{A,t}(x):=\ind_{A\times\{t\}\leftrightarrow(x,t+s)},\quad s\geq 0, x\in\mathbb{Z}.
\end{equation}
The process $\eta^{A,t}$ is called the contact process started from $A$ at time $t$.
Similarly we define its dual $\hat{\eta}^{A,t}$
\begin{equation}
\hat{\eta}_s^{A,t}(x):=\ind_{A\times\{t\}\leftrightarrow(x,t-s)},\quad s\geq 0, x\in\mathbb{Z}.
\end{equation}
One can show that $\eta^{A,t}$ and $\hat{\eta}^{A,t}$ have the same distribution, that is, the contact process is self-dual.
Also note that
\begin{equation}
\label{e:self-duality}
\eta_s^{\mathbb{Z},t}(x)=\ind_{\hat{\eta}^{\{x\},t+s}_s\neq\underline{0}},
\end{equation}
where $\underline{0}$ stands for the configuration in $S^{\mathbb{Z}}$ whose all coordinates are null.
Finally we define the contact process with upper invariant measure by
\begin{equation}
\eta_t(x):=\ind_{\hat{\eta}_s^{\{x\},t}\neq \underline{0}\ \forall s\geq 0}.
\end{equation}
We want to prove a decoupling inequality for the environment $\eta$  when $\lambda>\lambda_c$, the critical infection rate (i.e. above which $\eta$ is not identically $\underline{0})$.
For that, the following classical result is going to be useful.
\begin{proposition}[\cite{durrett-griffeath83}]
  \label{p:survive_then_die}
  For every $\lambda > \lambda_c$ there exists $c=c(\lambda)>0$ such that
  \begin{equation}
    \mathbb{P}[\eta^{\{0\},0}_t \neq \underline{0}, \text{ but $\, \eta^{\{0\},0}_s = \underline{0}$\, for some $s > t$}] \leq c^{-1} \exp\{ -c t \},
  \end{equation}
  for every $t\geq 0$.
\end{proposition}

We now prove the decoupling inequality for the contact process in the upper invariant measure.
\begin{proposition}
  For any $\lambda > \lambda_c$ and $\alpha > 0$, the process $\eta_t$ satisfies $\mathscr{D}_{\mathrm{env}}(\uc{c:decouple}, \alpha)$ for some $\uc{c:decouple} > 0$.
\end{proposition}

\begin{proof}
  Fix $t, r > 0$ and let $B_2 \subset \mathbb{L}$ be a box of the form $[z, z + 5 r] \times [t + r, t + 6r] \cap \mathbb{L}$. 
  Let also $f_1$ and $f_2$ be two functions measurable with respect to $\sigma(\eta_s(x); s \leq t, x \in \mathbb{Z})$ and $ \sigma(\eta_s(x); (x, s) \in B_2)$ respectively and such that $f_i (\cdot) \in [0,1]$ for $i=1,2$.
It is enough to bound the covariance between functions of this type.

  Now, let us introduce another process $(\eta'_s)_{s \geq t}$ as follows
  \begin{equation}
    \eta'_s(x) = \eta^{\mathbb{Z},t}_{s-t}(x)=\ind_{\hat{\eta}^{\{x\},s}_{s-t}\neq \underline{0}},
  \end{equation}
   the contact process started from the fully infected configuration at time $t$, or in other words, the set of points $(x,s)$ in  the ``semi-plane'' $s\geq t$ from which the dual process survives down to time $t $.

  Note that the process $(\eta'_s)_{s \geq t+ r}$ is independent of $(\eta_s)_{s \leq t}$.
  Therefore we can bound
  \begin{equation*}
    \begin{split}
      \Cov(f_1, f_2) & \leq \mathbb{P} [\eta'_s(x) \neq \eta_s(x) \text{ for some $(x, s) \in B_2$}]\\
      & \leq \mathbb{P} \big[ \eta'_{t + r}(x) \neq \eta_{t + r}(x) \text{ for some } x \in [z - 100 r, z + 100 r] \big]\\
      & \quad + \mathbb{P} \big[ \exists\, y \not \in [z - 100 r, z + 100 r],\ \exists\, (x,s)\in B_2\ \colon\ (y,t+r)\leftrightarrow(x,s) \big]\\
      & \leq c \exp\{ - c' r \},
    \end{split}
  \end{equation*}
  where in the last inequality we have used a simple large deviations estimate for a Poisson random variable to bound the second term, and the self-duality property \eqref{e:self-duality} together with Proposition~\ref{p:survive_then_die} to bound the first term.
\end{proof}

\subsection{Systems with a positive spectral gap}
\label{ss:spectral}

In this section we treat environments that satisfy a few hypotheses falling into the $L^2$-theory of stochastic processes.
More precisely, let us assume that
\begin{enumerate}[\quad a)]
\item $(\eta_t)_{t \geq 0}$ is a c\`adl\`ag Markov process on $S^\mathbb{Z}$.
\item $(\eta_t)_{t \geq 0}$ has a stationary measure $\nu$ and a semi-group $(S_t)_{t\geq 0}$ satisfying $S_tf(\eta)=\mathbb{E}_\eta[f(\eta_t)]$ that is strongly continuous in $L^2(\nu)$.
\item The generator $L$ of the process has a positive spectral gap.
\end{enumerate}

The following result is standard.

\begin{proposition}
  Under the above hypotheses, there exists $\beta > 0$ such that
  \begin{equation}
    \label{e:gap_decouple}
    \Big\lVert S_t(f) - \int f d \nu \Big\rVert_\nu \leq e^{-\beta t} \lVert f \rVert_\nu,
  \end{equation}
  where $\lVert \cdot \rVert_\nu$ is the $L^2$-norm associated with $\nu$.
\end{proposition}

We now show that
\begin{proposition}
Let $\eta=(\eta_t)_{t\geq 0}$ be a Markov process with stationary measure $\nu$.
Assume that $\eta$ satisfies \eqref{e:gap_decouple}.
Then for any $\alpha>0$, there exists $\uc{c:decouple} > 0$ for which $\eta$ satisfies $\mathscr{D}_{\mathrm{env}}(\uc{c:decouple}, \alpha)$.
\end{proposition}

\begin{proof}
For an interval $I\subset \mathbb{R}$, let $D(I,S^\mathbb{Z})$ be the set of c\`adl\`ag paths from $I$ to $S^\mathbb{Z}$ and let us abbreviate $\eta_I:=(\eta_s)_{s\in I}$. 
Let $r\geq 1$, $T\geq 0$, $f_1:D([0,T],S^\mathbb{Z})\rightarrow[-1,1]$ and $f_2:D(\mathbb{R}_+,S^\mathbb{Z})\rightarrow[-1,1]$ with $\mathbb{E}_\nu[f_2(\eta)]=0$. 
It is enough to show that there exists $\beta>0$ such that for any such choice of $r,T,f_1,f_2$
  \begin{equation}\label{e:decouplespectralgap}
  \left|\mathbb{E}_\nu\left[f_1\left(\eta_{[0,T]}\right)f_2\left(\eta_{[T+r,+\infty)}\right)\right]\right|\leq e^{-\beta r}.
  \end{equation}
For $\eta_0\in S^\mathbb{Z}$, let $\tilde{f}_2(\eta_0):=\mathbb{E}_{\eta_0}\left[f_2\left(\eta_{\mathbb{R}_+}\right)\right]$. Note that $\nu(\tilde{f}_2)=0$. By the Markov property applied at times $T$ and $T+r$, the left-hand term in \eqref{e:decouplespectralgap} can be rewritten and bounded as follows
  \begin{multline}
  \Big|\mathbb{E}_\nu\big[f_1\left(\eta_{[0,T]}\right)S_r\tilde{f}_2(\eta_{T})\big]\Big|
  \leq \mathbb{E}_\nu\Big[\big(S_r\tilde{f}_2(\eta_{T})\big)^2\Big]^{1/2}=\nu\Big(\big(S_r\tilde{f}_2\big)^2\Big)^{1/2}\leq e^{-\beta r}.
  \end{multline}
\end{proof}
Above we used stationarity of $\nu$ for the equality sign and \eqref{e:gap_decouple} in the last inequality.

Examples that fall into this class are:
\begin{enumerate}[\quad a)]
\item Independent spin flip dynamics.
\item Glauber dynamics for the Ising model in $\mathbb{Z}$, see \cite{liggett}, Corollary~4.18, p. 210,
\item The ``East model'' that will be discussed in more detail in Section~\ref{ss:east}, see \cite{Aldous2002}.
\end{enumerate}

\subsection{The East model and its distinguished zero}
\label{ss:east}

The East model is a Markov process on $\lbrace 0,1\rbrace ^\mathbb{Z}$ that can be described as follows. Fix a parameter  $\rho\in(0,1)$. With rate one, each site tries to update: to a $1$ (occupied site) with probability $\rho$ and to a $0$ (empty site) with probability $1-\rho$. The update is successful at a site $x$ if and only if $x+1$ is empty at the time of the update. More formally, the generator of the process is given by
\begin{equation}
Lf(\eta)=\sum_{x\in\mathbb{Z}}(1-\eta(x+1))(\rho\eta(x)+(1-\rho)(1-\eta(x)))\left[f(\eta^x)-f(\eta)\right],
\end{equation}
where $\eta^x$ denotes the configuration $\eta$ flipped at $x$. The process can be constructed in the following way: attach to all sites in $\mathbb Z$ independent parameter $1$ Poisson processes; think of them as clocks ringing at exponential times to signal update possibilities. With every clock ring, associate independently a Bernoulli variable with parameter $\rho$. When a clock rings at $x$, check the state of its right neighbour $x+1$. If it is occupied, nothing changes; if it is empty, the configuration at $x$ is refreshed using the Bernoulli variable associated with the ring. In the latter case, the ring is called \emph{legal}.

This process was introduced in the physics literature \cite{eastphys} to model the glass transition. It is not difficult to check that the product Bernoulli measure $\nu=\mathrm{Ber}(\rho)^{\otimes\mathbb{Z}}$ is reversible for this dynamics. Moreover, it was shown in \cite{Aldous2002} that the East model has positive spectral gap at any density $\rho\in(0,1)$.

One of the tools that are useful in the study of this model is the so-called \emph{distinguished zero} \cite{Aldous2002}, a c\`adl\`ag process on $\mathbb{Z}$ which we now describe. Recall the graphical construction given above. Start the process from a configuration with a $0$, say at the origin, and call it the distinguished zero. Wait for the first legal ring at the origin. By definition of a legal ring, before that time the configuration does not change at the origin and we let the distinguished zero remain there. Again by definition, at the time of the first legal ring, the site $1$ is empty; we then make the distinguished zero jump one step to the right. Then we iterate the construction: the distinguished zero remains at $1$ until a legal ring occurs there and then jumps to site $2$. 
See Figure~\ref{fig:distzero} for an illustration.

\begin{figure}
\begin{center}
\includegraphics[scale=.5]{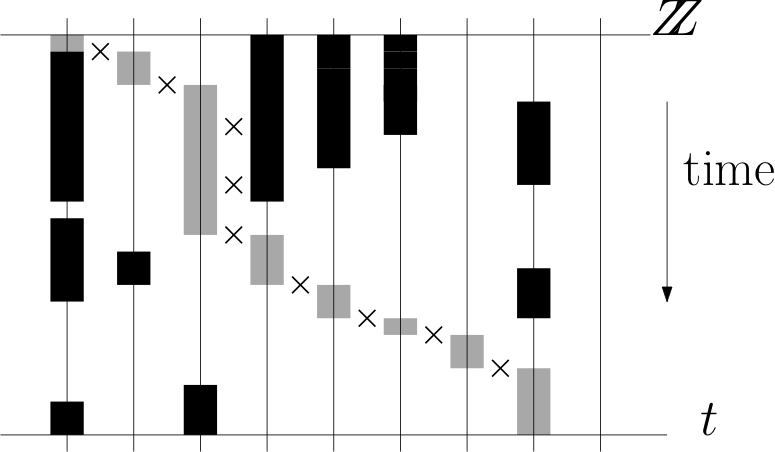}
\caption{\cite{frontprogression} In grey, a trajectory of a distinguished zero up to time t; time goes downwards, sites
are highlighted in black at the times when they are occupied. The crosses represent clock rings at the site occupied by the
distinguished zero.}
\label{fig:distzero}
\end{center}
\end{figure}
The interest of this object is that, because information travels from right to left in the East model, the distinguished zero acts as a buffer between the dynamics on its left and right. More precisely, conditional to the distinguished zero starting from the origin being at $x$ at time $t$, the distribution of the configuration on $\lbrace 0,\ldots,x-1\rbrace$ is exactly $\mathrm{Ber}(\rho)^{\otimes x}$, no matter what the initial configuration was (which had a zero at the origin). One application of our results is that the distinguished zero travels with positive speed to the right. 

\begin{proposition}\label{prop:distzero}
Start the East process with a distribution whose marginal on $\mathbb{N}$ is $\mathrm{Ber}(\rho)^{\otimes\mathbb N}$ and which puts a zero at the origin. Let $\xi_t$ be the position at time $t$ of the distinguished zero started from the origin. Then, a.s.\@
\begin{equation}
\frac{\xi_t}{t}\underset{t\rightarrow\infty}{\longrightarrow}v_d(\rho)>0.
\end{equation}
\end{proposition}

We also partially recover a result of \cite{frontprogression,GLM15}.
\begin{proposition}\label{prop:front}
Start the East process from the product Bernoulli distribution with parameter $1$ on $-\mathbb{N}$, $0$ on the origin and $\rho$ on $\mathbb N$. Let $X_t$ be the position of the leftmost zero at time $t$. 
Then a.s.\@
\begin{equation}
\frac{X_t}{t}\underset{t\rightarrow\infty}{\longrightarrow}v_f (\rho)<0.
\end{equation}
\end{proposition}

Both these results follow from Theorems~\ref{t:main} and \ref{t:positive}, together with the orientation property of the East model. 
Namely, in the East process, the distribution on $\lbrace x,x+1,\ldots\rbrace$ depends only on the initial distribution on $\lbrace x,x+1,\ldots\rbrace$. This follows immediately from the graphical construction.

\begin{proof}[Proof of Proposition~\ref{prop:distzero}]
Let us first give an alternative definition of the process $(\xi_t)_{t\geq 0}$. No matter what the initial configuration is, $\xi_0=0$. Then if the process sits at $x$ at a given time, we wait for the first legal ring at $x$, at which time $\xi$ jumps to the right. It is easy to check that this process coincides with the distinguished zero if there is a zero at the origin in the initial configuration. The upside of this formulation is that it fits in our setting (where the environment is stationary and cannot have a $0$ at the origin almost surely). The $T_i^x$ are given by the Poisson clocks on each site and $g$ is given by $g(\eta_{T_i^x}(x),\eta_{T_i^x}(x+1),U_i^x)=1-\eta_{T_i^x}(x+1)$. Moreover, thanks to the orientation property of the East model mentioned immediately above, the law of $\xi$ depends only on the marginal distribution of the initial measure on $\mathbb{N}$. Therefore, we can apply Theorems~\ref{t:main} and \ref{t:positive} to get the result (condition~\eqref{e:one_jump} is easy to check).

\end{proof}

\begin{proof}[Proof of Proposition~\ref{prop:front}]
Similarly to the previous proof, we give an alternative definition of the front process $X_t$. 
No matter what the initial configuration is, $X_0=0$. 
Then if the process sits at $x$ at a given time, we wait for the first legal ring either at $x$ or $x-1$. 
If it happens at $x$ and the associated Bernoulli variable is $1$, we let $X$ jump one step to the right. 
If it happens at $x-1$ and the associated Bernoulli variable is $0$, we let $X$ jump one step to the left. 
Else $X$ remains at $x$ and we wait for a new legal ring. 
It is not difficult to check that this process coincides with the front process when the initial configuration is as in Proposition~\ref{prop:front}. 
The orientation property together with a simple adaptation (see Remark \ref{r:front} below) then allows us to apply the results of Theorems~\ref{t:main} and ~\ref{t:positive} to this problem.
\end{proof}

\begin{remark}
\label{r:front}
Strictly speaking, the process described in the proof of Proposition \ref{prop:front} does not exactly fit in our setting as described in Section \ref{ss:randomwalker} since the jumping times $T_i^x$ should now be given by the superposition of the clocks at $x$ and $x-1$ and the jumps depend on which of these two clocks has actually rung.
Although we could generalize our setting in Section \ref{ss:randomwalker} in order to accommodate for this more general situation, for the sake of simplicity, we prefer to leave the standard adaptations for the interested reader.
\end{remark}

\subsection{Independent renewal chains}
\label{ss:renewal}

Let us fix a sequence $(a_n)_{n \in \mathbb{N}}$ of positive real numbers satisfying
\begin{equation}
  0 < a_n \leq c e^{-\log^2 n}, \text{ for every $n > 0$}
\end{equation}
and consider the induced probability distribution $p_n =(1/Z) a_n$, where $Z$ is the appropriate normalization constant.
We now consider the state space $S = \mathbb{Z}_+$ and define a renewal chain with renewal times given by $p_n$.
More precisely, consider a Markov process on $S$ evolving according to the following generator
\begin{equation}
  Lf(n) = \begin{cases}
  f(n - 1) - f(n) & \text{ if }n>0\\
  \sum_{k>0}p_k(f(k)-f(0))&\text{ if }n=0.
  \end{cases}
\end{equation}
Intuitively speaking, when the chain is at some site $n > 0$, it jumps with unitary rate one to $n - 1$.
At zero it also jumps with rate one to a random integer $n > 0$ with probability proportional to $p_n$.

It is not difficult to see that this chain has stationary distribution given by
\begin{equation}
  q_n = \frac{1}{Z'} \sum_{j \geq n} a_n, \quad Z' = \sum_{n > 0} \sum_{j \geq n} a_n.
\end{equation}

For each site $x \in \mathbb{Z}$, we independently let $\eta_t(x)$ evolve as the above chain starting from the stationary measure, thus defining the dynamic random environment.
To see that this chain satisfies $\mathscr{D}_{\mathrm{env}}(\uc{c:decouple}, \alpha)$, we refer the reader to (3.47) of \cite{zbMATH06514478}.

Note that this chain is not uniformly mixing as observed in Remark~3.7 of \cite{zbMATH06514478}.

\section{Counterexample of an ergodic environment}
\label{s:example}

In this section we construct an example of random environment which is space-time ergodic, but such that very natural random walks can be constructed on top of it without obeying a LLN.
In our example all hypotheses of Theorem~\ref{t:main} will be satisfied except that the environment fails to fulfil the decoupling property $\mathscr{D}_{\mathrm{env}}(\uc{c:decouple}, \alpha)$ which is replaced by the weaker ergodicity property.
Therefore, $\mathbb{P}$ does not satisfy $\mathscr{D}(\uc{c:decouple}, \alpha)$.
This example should serve as a cautionary tale for the difficulties in going from ergodicity of the space-time environment to ergodicity of the environment as viewed from the random walker.

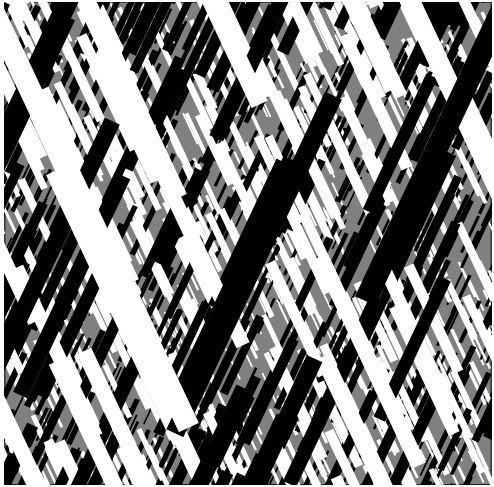
\begin{figure}[ht]
  \centering
  \begin{tikzpicture}[scale=2]
    \begin{scope}
      \clip (-.6, -.6) rectangle (2.6, 2.6);
      \iffinal
      \draw[fill=gray] (-.6, -.6) rectangle (2.6, 2.6);
      \foreach \z in {900,...,1}
      { \pgfmathrandominteger{\x}{0}{200};
        \pgfmathrandominteger{\y}{0}{200};
        \pgfmathrandominteger{\c}{1}{1000};
        \pgfmathrandominteger{\d}{0}{1};

        \pgfmathsetmacro{\comprimento}{.1 * exp(-.3* log2(.001 * \z))}%
        \pgfmathsetmacro{\largura}{-.01 *  ln(.001 * \z) }%

        \draw[color=black,fill=black,
              rotate around={{65}:({.02*\x-1}, {.02*\y-1})}]
                  ({.02*\x-1 - \comprimento}, {.02*\y-1 - \largura})
        rectangle ({.02*\x-1 + \comprimento}, {.02*\y-1 + \largura});
        \pgfmathrandominteger{\x}{0}{200};
        \pgfmathrandominteger{\y}{0}{200};
        \pgfmathrandominteger{\c}{1}{1000};
        \pgfmathrandominteger{\d}{0}{1};
        \pgfmathsetmacro{\comprimento}{.1 * exp(-.3* log2(.001 * \z))}%
        \pgfmathsetmacro{\largura}{-.01 *  ln(.001 * \z) }%

        \draw[color=white,fill=white,
              rotate around={{-65}:({.02*\x-1}, {.02*\y-1})}]
                  ({.02*\x-1 - \comprimento}, {.02*\y-1 - \largura})
        rectangle ({.02*\x-1 + \comprimento}, {.02*\y-1 + \largura});
      }
      \fi
    \end{scope}
  \end{tikzpicture}
  \caption{An illustration of the environment built in our counter example.}
   \label{f:confetti}
\end{figure}

The random environment we construct assigns the colors \emph{black}, \emph{white} or \emph{gray} to every point in the plane, see Figure~\ref{f:confetti}.
These colors will influence the local drift experienced by the random walker.
We will give a precise definition of the jump rates in a moment, but for now it is enough to know that black sites will induce a drift to the right, regions in white will create a drift to the left and gray areas will induce symmetric jumps fro the walker.

The construction of this environment is based on a colored continuum percolation, in which random obstacles (which are colored either black or white) are placed on top of each other.
Although this model resembles confetti percolation, there are differences that make us prefer to describe it as a colored continuum percolation.
The colored obstacles are going to be tilted rectangle with random side lengths, see Figure~\ref{f:confetti} for an illustration.

In this Poissonian soup of rectangles, black rectangles will be tilted towards the right and will induce a positive drift on the random walker.
On the contrary, white rectangles will be tilted to the left and induce a negative drift on the walker.
Finally, regions that are not covered by any rectangle will be declared gray.

Roughly speaking, when the random walk hits large monochromatic regions, it will experience a strong drift for a long time.
This behavior will ultimately result in linear fluctuations on the displacement of the random walker's trajectory, therefore breaking the Law of Large Numbers.

\vspace{3mm}

We start to make our construction precise by defining the sizes of these rectangles.
For this we introduce the following sequence of scales:
\begin{equation}
  \begin{array}{c}
    L_0 = 10^5, \qquad l_k = L_k^{1/5}, \qquad L_{k + 1} = l_k L_k, \text{ for integers $k \geq 0$}.
  \end{array}
\end{equation}
The above choices are somewhat arbitrary, but enough for the purpose of this section.

For each $k \geq 0$, we build a homogeneous Poisson process in $\mathbb{R}^2$ with density $L_k^{-2}$. 
The Poisson processes are assumed to be independent of one another for different values of $k$.
Given a scale $k$, we are going to decorate each of the points $y$ in the support of  corresponding Poisson Process with a rectangle as follows:
\begin{enumerate}
\item A rectangle corresponding to the point $y$ will be centered at $y$.
\item It will have length $L_k$ and width $\log^2(L_k)$.
\item It will be assigned independently colors \emph{black} or \emph{white} with equal probability.
\item The rectangle's longest axis will form an angle of $-30$ degrees with the vertical axis when the rectangle is black and of $30$ if it is white.
\item It will be assigned independently a uniform random variable in $[0, 1]$ called its height.
It will simply be used in order to break ties.
\end{enumerate}

We first calculate the probability that a certain rectangle resulting from the $k$-th Poisson Process touches the origin.
The number of rectangles touching the origin is a Poisson random variable with parameter bounded by
\begin{equation}
  c\, L_k \cdot \log^2(L_k)\cdot L_k^{-2},
\end{equation}
which is summable in $k$.
Therefore, almost surely, only finitely many scales can influence a given point.
Or in other words, almost surely each point is covered by finitely many rectangles.

We can now define the environment in which our random walker will evolve.
Points in the plane which are not covered by any rectangle are colored gray.
On the other hand if $x \in \mathbb{R}^2$ is covered by at least one rectangle, we color $x$ with the color of the largest rectangle that covers $x$.
In case there are various rectangles at the same scale that cover $x$ we break ties using their heights.

Having fully described the environment, we now define the law of the random walker that evolves on top of this environment.
The random walk $Y_t$ will always belong to $\mathbb{L} = \mathbb{Z} \times \mathbb{R}_+$ and will jump with rate one to a neighboring site with probabilities that depend on the color of the environment at the moment of jump:
\begin{enumerate}
\item If the environment at the moment of jump is gray, the jump is made symmetric,
\item If it is is black, the walker jumps with probability $0.9$ to the right and $0.1$ to the left.
\item If it is white, the walker jumps with probability $0.1$ to the right and  $0.9$ and to the left.
\end{enumerate}

\nc{c:touch_box}
We now want to show that the random walker exhibits linear fluctuations, ruling out a LLN.
This is done by showing that in times of order $L_k$ there is a positive probability that a unique rectangle of size $L_k$ crosses the way of the random walker, completely determining its direction.

We start by showing that there is a positive probability that a black rectangle of scale $k$ touches the sets $A = [-2 \epsilon L_k, \epsilon L_k] \times \{0\}$ and $B = \mathbb{R} \times \{1/2 L_k\}$.
In fact, there exists a positive constant $\uc{c:touch_box}$ such that
\begin{equation}
  \label{e:single_A_B}
  \begin{split}
    \mathbb{P} \Big[ & \text{a single rectangle at scale $k$ touches $A$ and $B$ and it is black} \Big]\\
    & \geq c \; \mathbb{P} \Big[ \text{a single rectangle at scale $k$ touches $A$} \Big]
    \geq \uc{c:touch_box}, \text{ for all $k \geq 0$}.
\end{split}
\end{equation}
Indeed, the number of such rectangles at scale $k$ is a Poisson random variable with parameter of order one: up to multiplicative constants it is $L_k^2 \cdot L_k^{-2} = 1$.

Now, for a fixed $\bar{k}$, the probability that a rectangle of a scale $k > \bar{k}$ touches $[0,L_{\bar{k}}]^2$ can be controlled by considering the area of the sumset of the rectangle $[0,L_{\bar{k}}]^2$ and $[0,L_k]\times [0,\log^2(L_k)]$ :
\begin{equation}
  \label{e:no_larger}
  \begin{split}
    \mathbb{P} \Big[ \text{some }
    & \text{rectangle from scale $k>\bar{k}$ touches $[0, L_{\bar{k}}$}]^2 \Big] \\
    & \leq c \sum_{k>\bar{k}} \big(L_{\bar{k}} +\log^2 (L_k) \big)(L_{\bar{k}} +L_k) L_k^{-2}
    \leq c \sum_{k>\bar{k}} \frac{L_{\bar{k}}}{L_k} + \frac{\log^2(L_k)}{L_k} \\
    & \leq c \sum_{k > \bar{k}} \frac{1}{l_{\bar{k}} \cdots l_{k - 1} } + c \sum_{k \geq \bar{k}} \frac{1}{\sqrt{L_k}}.
  \end{split}
\end{equation}
Note that the above converges to zero as $\bar{k}$ goes to infinity, meaning that as the scale $\bar{k}$ of a box grows there is a constant probability that it will intersect a single rectangle at scale $\bar{k}$ and no larger rectangles.

If the event in \eqref{e:single_A_B} occurs, but not the event in \eqref{e:no_larger}, then there is a positive probability that $X_{L_{\bar{k}}}$ has a displacement larger than $L_{\bar{k}}/10$.
By symmetry, there is also a positive probability that $X_{L_{\bar{k}}}$ is smaller than $-L_{\bar{k}}/10$.
This rules out the possibility that the random walker satisfies a Law of Large Numbers.

All that is left to complete our counter-example is to show that the above random environment is ergodic in space-time.

For this, we consider two boxes $B_1$ and $B_2$ within vertical distance $r$ and having side length $r^\alpha$.
We will show that for any two functions $f_1$ and $f_2$ with $\lVert f_i \rVert \leq 1$ and that only depend on what happens inside these two boxes, the covariance between them is bounded by something that goes to zero with $r$.
This statement implies that the environment is space-time mixing and therefore ergodic.

To prove the above covariance bound we will use a technique similar to what is done in \cite{Ahlberg2017} for boolean percolation.
Indeed, the proof of Proposition~2.2 of \cite{Ahlberg2017} shows that we can bound the covariance of $f_1$ and $f_2$ as follows:
\begin{equation}
  \Cov(f_1, f_2) \leq 4 \mathbb{P} \Big[ \text{some rectangle touches both $B_1$ and $B_2$} \Big].
\end{equation}
The above probability can be bound by summing over all scales of rectangles that are long enough to touch both boxes.
Let $\bar{k}$ be such that $L_{\bar{k} - 1} < r \leq L_{\bar{k}}$ and estimate:
\begin{equation}
  \begin{split}
    \Cov(f_1, f_2) & \leq c \sum_{k \geq \bar{k}} (L_k + r^\alpha) (\log^2(L_k) + r^\alpha) L_k^{-2}\\
    & \overset{r^\alpha \leq L_{\bar{k}}}\leq c \sum_{k \geq L_{\bar{k}}} \frac{\log^2(L_k)}{L_k} + c\sum_{k \geq L_{\bar{k}}} \frac{L_{\bar{k}}^\alpha}{L_k},
  \end{split}
\end{equation}
which clearly converges to zero as $\bar{k}$ goes to infinity.
This shows that the environment is mixing and therefore ergodic.

\begin{remark}
As one can inspect the above covariance estimate would not converge to zero if the boxes $B_1$ and $B_2$ were taken to be of size $r$ as in $\mathscr{D}_{env}(\uc{c:decouple}, \alpha)$.
\end{remark}

\begin{remark}
  It is currently an open question whether the LLN for the random walker follows by simply assuming that the covariance in Definition~\ref{d:envdecouple} goes to zero with~$r$.
\end{remark}

\section*{Acknowledgements}
This work has been supported by the ANR
projects LSD (ANR-15-CE40-0020) and MALIN (ANR-16-CE93-0003), and by the LABEX MILYON (ANR-10-LABX-0070) of Universit\'e de
Lyon, within the program ``Investissements  d'Avenir" (ANR-11-IDEX-0007) operated by
the French National Research Agency (ANR).
MH was supported by CNPq grants ``Projeto Universal" (307880/2017-6) and ``Produtividade em Pesquisa" (406659/2016-8) and by FAPEMIG grant ``Projeto Universal" (APQ-02971-17).
AT was supported by grants ``Projeto Universal'' (406250/2016-2) and ``Produtividade em Pesquisa'' (309356/2015-6) from CNPq and ``Jovem Cientista do Nosso Estado'', (202.231/2015) from FAPERJ.
Part of this collaboration took place in CIB (Lausanne), MH and AT thank the staff for the support and hospitality. OB and AT visited The University of Geneva where MH was a long-term research fellow during the beginning of this collaboration. OB and MH acknowledge IMPA for its support and hospitality during several visits.

\bibliographystyle{alpha}
\bibliography{all}

\end{document}